\documentclass[11pt]{amsart}
\usepackage[latin1]{inputenc}
\usepackage{pdfsync}
\usepackage[english]{babel}

\usepackage[colorlinks=true,urlcolor=blue,linkcolor=black,citecolor=black]{hyperref}
\usepackage{graphicx}

\usepackage[capitalize]{cleveref}
\usepackage{fullpage}
\usepackage{cite}
\usepackage{color}
\usepackage{amsmath}
\usepackage{amssymb}
\usepackage{amsfonts}
\usepackage{mathrsfs}
\usepackage{verbatim}
\usepackage{bbm}
\usepackage{enumerate}

\usepackage[left= 1.25 in, right= 1.25 in ,top= 1 in, bottom = 2 in]{geometry}
\newcommand{\R}{\mathbb{R}}

\newcommand{\Z}{\mathbb{Z}}
\newcommand{\N}{\mathbb{N}}

\renewcommand{\P}{\mathbb{P}}

\renewcommand{\mod}[1]{\text{ (mod #1)}}
\newcommand{\sse}{\subseteq}

\newcommand{\AVG}{\frac{1}{N} \sum_{n\leq N}}

\newcommand{\norm}[1]{\left\lVert#1\right\rVert}
\newcommand{\floor}[1]{\lfloor #1 \rfloor}

\newtheorem{theorem}{Theorem}[section]
\newtheorem*{theorem*}{Theorem}
\newtheorem{prop}[theorem]{Proposition}
\newtheorem{lemma}[theorem]{Lemma}

\newtheorem{corollary}[theorem]{Corollary}
\newtheorem*{corollary*}{Corollary}

\theoremstyle{definition}
\newtheorem*{definition*}{Definition}
\newtheorem{defn}[theorem]{Definition}

\theoremstyle{remark}

\newtheorem{remark}[theorem]{Remark}
\newtheorem*{remark*}{Remark}

\makeatletter
\def\thmhead@plain#1#2#3{%
  \thmname{#1}\thmnumber{\@ifnotempty{#1}{ }\@upn{#2}}%
  \thmnote{ {\the\thm@notefont#3}}}
\let\thmhead\thmhead@plain
\makeatother
\theoremstyle{definition}

\title{Visible Measures along \( \Omega(n)\) and Distribution of Horocycle Orbits}
\date{\today}

\author{Adam Kanigowski and Kaitlyn Loyd}

\thanks{No generative AI tools were used in producing the ideas, proofs, or text appearing in this manuscript.}

\begin{document}

\begin{abstract}
    Let \( \Omega(n) \) denote the number of prime factors of \( n \), counted with multiplicities. We study the set \(Acc^\Omega(x)\) of weak-$^*$ limits of the sequence $\frac{1}{N}\sum_{n\leq N}\delta_{T^{\Omega(n)}x}$ in \( \sigma\)-compact dynamical systems \( (X,T)\), demonstrating that if \( x \in X \) is quasi-generic for an ergodic measure \( \mu \), then \( \mu \in Acc^\Omega(x)\). This extends a result of Bergelson and Richter, who studied the problem in the setting of uniquely ergodic systems. 
    
    We give a more precise description of the set $Acc^\Omega(x)$ in the case of the horocycle flow on non-compact quotients of $SL(2,\R)$. We show that for every non-periodic $x\in X$, in addition to Haar measure, there exists sequences $(s_n), (c_n) \sse \R$ such that 
    $$  
        \frac{1}{\sqrt{2\pi}}\int_{-\infty}^{\infty}e^{-\frac{r^2}{2}}\nu^{i}_{s_n-2\log|1+c_nr|} dr\in Acc^{\Omega}(x),
    $$ 
    where \( \{ \nu^{i}_{s} \}_{i \leq k}\) denotes the one parameter family of periodic measures in each of the \( k \) inequivalent cusps. Depending on Diophantine properties of the non-periodic point \( x\), we show that \( Acc^\Omega(x)\) contains a full two parameter family of such periodic measures, as well as the Dirac measure at each cusp. In particular, these results yield almost-everywhere divergence of pointwise averages along \( \Omega(n) \) for the non-compact horocycle flow.

\end{abstract}

\maketitle

\section{Introduction}

    Let $X$ be a $\sigma$-compact metric space and $T: X\to X$ a continuous map. In this paper, we are interested in the asymptotic behavior of the sequence $(T^{\Omega(n)}x)_{n \in \N}$, where $x\in X$ and $\Omega(n)$ counts the number of prime factors of $n$ with multiplicities. Let $M(X,T)$ (respectively $M^e(X,T)$) denote the set of $T$-invariant Borel probability measures on $X$ (respectively $T$-ergodic Borel probability measures). When $M(X,T)$ is a singleton, the dynamical system \((X,T)\) is called uniquely ergodic. A natural way to study orbits sampled along $\Omega(n)$ is to determine possible weak-$^*$ accumulation points of the sequence $\frac{1}{N}\sum_{n\leq N} \delta_{T^{\Omega(n)}x}$. A landmark result of Bergelson and Richter \cite{BR2020}, states that if $(X,\mu, T)$ is uniquely ergodic and $X$ is compact, then for every $x\in X$, 
    \[
        \lim_{N \to \infty}\frac{1}{N}\sum_{n\leq N} \delta_{T^{\Omega(n)}x} =  \mu.
    \]
    The main focus of this paper is therefore on non-uniquely ergodic transformations and non-compact spaces. 
    
    For $x\in X$, let $Acc(x)$ denote the set of weak-$^*$ accumulation points of $\AVG \delta_{T^n x}$. We refer to elements of $Acc(x)$ as visible measures. Similarly, let $Acc^\Omega(x)$ denote the set of weak-$^*$ accumulation points of $\AVG \delta_{T^{\Omega(n)}x}$. Using this terminology, Bergelson and Richter's result states that if $(X, \mu, T)$ is uniquely ergodic and $X$ is compact, then $Acc(x)=\{\mu\}=Acc^\Omega(x)$ for every $x\in X$. Further, it follows from the proof in \cite{BR2020} that, in any compact dynamical system, each measure in the set $Acc^\Omega(x)$ is $T$-invariant. Therefore, it is natural to study the relation between these two subsets of $M(X,T)$. Our first main result extends the Bergelson-Richter Theorem, demonstrating that any ergodic measure visible along standard Birkhoff averages is also visible along \( \Omega \)-averages: 
  
    \begin{theorem}
    \label{thm:main}
        Let \( X \) be a $\sigma$-compact metric space and \( (X, T) \) a dynamical system. Then 
        \[
            Acc(x) \cap M^e(X,T) \sse Acc^{\Omega}(x)
        \]
        for every \( x \in X\).
    \end{theorem}
    
    The assumption of ergodicity is necessary in \cref{thm:main}, and in \cref{sec:counterex}, we construct a symbolic counterexample for the non-ergodic case. In fact, we can construct the point \( x \in X\) to be generic for the chosen measure: 
    
    \begin{prop}
    \label{prop:counterex}
        There exists a compact dynamical system $(X,T)$, non-ergodic measure \( \nu \in M(X,T)\), and point $x\in X$ such that $Acc(x) = \{\nu\}$ and $\nu\notin Acc^\Omega(x)$.
    \end{prop}
    
    Theorem \ref{thm:main} shows in particular that if $(X,T)$ is a dynamical system such that every point $x\in X$ is generic\footnote{This by definition means that $Acc(x)$ is a singleton.} for an ergodic measure, then $Acc(x) \subseteq Acc^\Omega(x)$ for every $x\in X$. The main class of systems satisfying this condition is given by unipotent flows on quotients of semi-simple Lie groups, where the genericity assumption follows from fundamental results of Ratner \cite{Ratner}. Our second main set of results concern the relation between $Acc(x)$ and $Acc^\Omega(x)$ for the most studied class of unipotent flows: horocycle flows.

    Let $G=PSL(2,\R)$ and $\Gamma \subset G$ a lattice. Set $X=\Gamma\backslash G$ and let $\mu_X$ denote the Haar measure on $X$. Let 
    \[
        h_t=\begin{pmatrix} 1&t\\0&1\end{pmatrix}
    \]
    be the horocycle flow acting on $X$ by right multiplication. The discrete horocycle flow is given by the \( \Z \)-action of the time one map \( h = h_1 \) on \( X \). If $X$ is compact, it follows by a result of Furstenberg \cite{Furstenberg} that $(h_t)$ is uniquely ergodic. In this case, the result of \cite{BR2020} applies. Hence, we consider the non-compact case, in which it is known that for some $k\geq 1$, the space $X$ has $k$ inequivalent cusps, and for each cusp $C_i$, there is a one-parameter family of periodic measures corresponding to $C_i$ (see for instance \cite{D-S} or \cref{lem:perppoint}). Denote by $\{\nu^i_s\}_{s\in \R, i\leq k}$ the set of periodic measures for the horocycle flow on $X$. It follows from a result of Dani and Smillie \cite{D-S} that for every $x\in X$, $Acc(x)$ is a singleton, equal to either $\mu_X$ or a periodic measure. Moreover, let $\nu^i_\infty$ denote the Dirac measure at $\infty$ corresponding to $i$-th cusp. Note that \( \nu^i_\infty\) is a probability measure on the $k$-point compactification of $X$. Since \( Acc^\Omega(x) \sse M(X,T) \) for compact systems, every element of \( Acc^\Omega(x) \) is a convex combination of periodic measures, the volume, and the Dirac masses at infinity. In the following result, we show that the convex combinations that arise are restricted and given by Gaussian distributions:
    
    \begin{theorem}
    \label{thm:horo}
        Let \( (X, h) \) be the discrete horocycle flow on \( \Gamma \backslash G \). For any non-periodic point\footnote{In this paper, the terms periodic and non-periodic refer to the flow \( (h_t) \). The action of the discrete horocycle on a periodic orbit is then either a finite rotation or irrational circle rotation.} $x \in X$, we have the following: 
        \begin{equation}\label{eq:subs}
            Acc^{\Omega}(x)
            \subseteq 
            \{\mu_X\}
            \cup
            \{\nu^i_\infty\}_{i\leq k}
            \cup 
            \left\{\frac{1}{\sqrt{2\pi}}\int_{-\infty}^{\infty}e^{-\frac{r^2}{2}}\nu^{i}_{s_0-2\log|1+c_0(r-z_0)|} dr\right\}_{s_0,c_0,z_0 \in \R, i\leq k}.
        \end{equation}
        Moreover, there exists $D>0$ such that for every $i \leq k$ and $s_0 \in \R$, there exists\footnote{In fact, $C(\cdot)$ depends continuously on $s_0$.} $C(s_0)>0$ and $s \in [s_0 - \frac12,s_0+ \frac12]$ such that for some $c \in [C(s_0),D C(s_0)]$, we have 
        \begin{equation}\label{eq:subs'}
            \frac{1}{\sqrt{2\pi}}\int_{-\infty}^{\infty}e^{-\frac{r^2}{2}}\nu^{i}_{s-2\log|1+cr|} dr\in Acc^{\Omega}(x).
        \end{equation}
        Finally, $\mu_X \in  Acc^{\Omega}(x)$.
    \end{theorem}
    
    The difference between Equations \eqref{eq:subs} and \eqref{eq:subs'} is the following: Equation \eqref{eq:subs} states that any accumulation point belongs to the set on the right hand side, with any parameters $s_0,c_0,z_0\in \R$ and \(i\leq k\), though equality need not hold (see for instance \cref{thm:horoboun}). On the other hand, Equation \eqref{eq:subs'} states that given $s_0 \in \R$ and \( i \leq k\), one can find an accumulation point centered a bounded distance from \( \nu^i_{s_0}\).

    \begin{corollary} 
    \label{cor:propersubs}
        For every non-periodic $x\in X$, $Acc(x)$ is a proper subset of $Acc^\Omega(x).$
    \end{corollary}

    Note that for each periodic point $p \in X$, it follows from the Bergelson-Richter theorem that $Acc^{\Omega}(p)$ is a singleton equal to \( Acc(p) \). Then \cref{cor:propersubs} classifies non-periodic points as those for which \( Acc(x) \subsetneq Acc^\Omega(x) \). A result of Furlong \cite{Furlong25} explores how large the set \( Acc^\Omega(x) \setminus Acc(x) \) can be in symbolic systems. Taking \( (X, \sigma) \) to be the full shift on a finite alphabet, he demonstrates existence of a point \( x \in X \) generic for a Dirac measure such that \( Acc^\Omega(x) = M(X,\sigma)\), i.e. any \( \sigma\)-invariant measure is visible along \( \Omega(n)\).
       
    As another corollary of \cref{thm:horo}, we resolve an unanswered case regarding pointwise convergence along \( \Omega(n) \). In \cite{Loyd2021}, the second author showed that in any non-atomic ergodic system \( (X, \mu, T)\), there exists \( f \in L^\infty(\mu)\) such that the \( \Omega\)-ergodic averages \( \AVG f(T^{\Omega(n)}x )\) diverge almost everywhere. This contrasts the Bergelson-Richter theorem, which guarantees pointwise convergence for continuous functions along \( \Omega(n) \) in uniquely ergodic systems. The question remained whether pointwise convergence for continuous functions holds in an arbitrary ergodic system. However, \cref{thm:horo} implies that pointwise convergence fails in this setting:

    \begin{corollary}
    \label{cor:ptwise}
        Let \((X, \mu_X, h)\) be the discrete horocycle flow on \( X = \Gamma \backslash PSL(2,\R)\), where \( \mu_X \) denotes the Haar measure on \( X \). Then there exists a continuous function \( f \) with compact support such that the averages
        \[
            \AVG f(h_{\Omega(n)}x ) 
        \]
        diverge for \(\mu_X\)-almost every point \( x \in X \).
    \end{corollary}

    The following two results give a more precise description of which measures on the right hand side of Equation \eqref{eq:subs} lie in the set $Acc^\Omega(x)$. In particular, inclusion of the measures $\nu^i_\infty$ and $\nu^i_s$, which arise by taking $c_0=0$, turns out to depend on the Diophantine properties of $x\in X$. Let $(g_t)$ denote the geodesic flow on \( X \).

    \begin{theorem} 
    \label{thm:horoboun} 
        Let $x\in X$ be such that the orbit $\{g_s(x): s\in \R\}$ is bounded\footnote{It is known that this set has full Hausdorff dimension.}. Then 
        $$
            \nu^i_s\notin Acc^{\Omega}(x).
        $$
        for every $i\leq k$ and $s\in \R\cup \{\infty\}$.
    \end{theorem}
    
    \cref{thm:horoboun} shows that for a set of full Hausdorff dimension, periodic measures do not appear as accumulation points for $\Omega$-averaging and there is no escape of mass. On the other hand, there are points for which the behavior is very different: 
    
    \begin{theorem}
    \label{thm:horoubo}
        There exists a dense $G_\delta$-set of points $G \subset X$ such that for every $x \in G$,
        $$
             Acc^{\Omega}(x)
             =
            \{\mu_X\}
            \cup
            \{\nu^i_\infty\}_{i\leq k}
            \cup 
            \left\{\frac{1}{\sqrt{2\pi}}\int_{-\infty}^{\infty}e^{-\frac{r^2}{2}}\nu^{i}_{s_0-2\log|1+c_0(r-z_0)|} dr\right\}_{s_0,c_0,z_0 \in \R, i\leq k}.
        $$
        In other words, there is equality in Equation \eqref{eq:subs}.
    \end{theorem}
    
    A key ingredient in the proofs of each of these results is a transition to weighted ergodic averages. Leveraging the repetition of \( \Omega(n) \), we convert the ergodic averages to weighted averages of the form
    \[
        \frac{1}{\sum_{k =1}^{N} w_N(k)} \sum_{k =1}^{N} w_N(k) f(T^k x) .
    \]  
    These weight functions are largely supported in relatively small intervals, reflecting that the values of \( \Omega(n) \) concentrate in these windows. Specifically, the weight functions satisfy
    \begin{equation*}
        \frac{1}{\sum_{k=1}^N w_N(k)} \sum_{k \in I_N} w_N(k) 
        = 
        1 + o_N(1),
    \end{equation*}
    where \( I_N \) is an interval centered at \(\log\log N\) with length of order \( \sqrt{\log\log N}\). In the proof of \cref{thm:main}, we then construct a large measure set \(A\) consisting of points whose statistical behavior we control using ergodicity, as well as a subsequence \( (N_i)\) for which \( T^k x \in A \) for a large proportion of iterates in the windows \( I_{N_i} \). 
    
    In the proof of \cref{thm:horo}, we additionally require quantitative equidistribution results for the horocycle flow, including several quantitative estimates relating the distribution of horocycle orbits to the distance to periodic points. For instance, as part of our main technical lemma (see \cref{lem:fini}), we show that if \(x \) is a non-periodic point with quantitative bounds on the distance to a periodic point, a segment of the horocycle orbit \( \{h_t x\}_{t \in J}\) can be approximated by the sheared orbit of a nearby periodic point. These intervals \( J \) are chosen as subintervals of the interval \( I_N \) coming from the distribution of \( \Omega(n)\). We then use Fourier analysis to show that these sheared periodic orbits have the desired statistical properties.  \\

    {\bf Acknowledgments:} The idea behind \cref{thm:horoubo} was communicated to the authors by M. Einsiedler. The authors would like to thank M. Einsiedler for the insight. The authors would also like to thank Johannes Luber for a careful reading and number of useful corrections to the earlier version of the text.

\section{Preliminaries}

    \subsection{Measure-preserving Systems}
    \label{ss: MPS}
        By a \textit{topological dynamical system}, we mean a pair \((X,T)\), where \(X\) is a $\sigma$-compact metric space and \(T\) a continuous map on \(X\). A Borel probability measure \(\mu\) on \(X\) is called \textit{T-invariant} if \(\mu(T^{-1}A) = \mu(A)\) for all measurable sets \(A\). By the Bogolyubov-Krylov Theorem, every compact topological dynamical system admits at least one \(T\)-invariant measure. When a topological system \((X, T)\) admits only one such measure, \((X, T)\) is called \textit{uniquely ergodic}. When there is no ambiguity, we say the map \( T \) is uniquely ergodic. 
        
        By a \textit{measure-preserving dynamical system}, we mean a probability space \((X, \mathcal{B}, \mu)\), where \(X\) is a $\sigma$-compact metric space, \(\mathcal{B}\) the Borel \(\sigma\)-algebra on \(X\), and \( \mu \) a Borel probability measure, accompanied by a measure-preserving transformation \(T: X \to X\). We often omit the \(\sigma\)-algebra \(\mathcal{B}\) when there is no ambiguity. A measure-preserving dynamical system is called \textit{ergodic} if for any \(A \in \mathcal{B}\) such that \(T^{-1}A = A\), one has \(\mu(A) = 0\) or \(\mu(A) = 1\).   
        
        
        Let \(C_b(X)\) denote the space of continuous bounded functions on \(X\). A point \(x \in X\) is called \textit{generic for the measure \(\mu\)} if 
        \[
            \lim_{N \to \infty} \AVG f(T^n x) = \int f \, d\mu
        \]
        for all \(f \in C_b(X)\) and \textit{quasi-generic for \(\mu\)} if there exists a subsequence \( (N_i)_{i \in \N} \sse \N\) such that 
        \[
            \lim_{i \to \infty} \frac{1}{N_i} \sum_{n=1}^{N_i}f(T^n x) = \int f \, d\mu
        \]
        for all \(f \in C_b(X)\). In the language of \cref{thm:main}, generic points are those for which $Acc(x)$ is a singleton and a point is quasi-generic for \( \mu \) if \( \mu \in Acc(x) \). When \(\mu\) is ergodic, the set of generic points has full measure and when \( \mu \) is uniquely ergodic, every point is generic.\\

    \subsection{Symbolic Systems}
    \label{subsec: Symbolic Dynamics}
        Let $\mathcal{A}$ be a finite set of symbols. Let $\mathcal{A}^\Z$ denote the set of all biinfinite sequences with entries in $\mathcal{A}$. The set $\mathcal{A}^\Z$ is endowed with the product topology coming from the discrete topology on the alphabet $\mathcal{A}$. Denote an element in  $\mathcal{A}^\Z$ by $\mathbf{x} = (x_n)_{n \in \Z}$. One equivalent choice of metric generated by this topology is given by
        \[
            \text{d}(\mathbf{x}, \mathbf{y}) = 2^{-\min \{|n| \,:\, x_n \neq y_n\}}. 
        \]
        This space carries a natural homomorphism $\sigma: \mathcal{A}^\Z \to \mathcal{A}^\Z$, called the \textit{left shift}, defined by \( (\sigma \mathbf{x})_n = x_{n+1} \).

    \subsection{Convergence of Measures}
    \label{ss:measures}
        Let $X$ be a $\sigma$-compact metric space and let $M(X)$ denote the set of probability measures on $X$. For a continuous function \( f\) and \( \mu \in M(X) \), define
        \[
            \mu(f) := \int f \, d\mu.
        \]
        Let \( C_c(X) \) denote the set of  compactly supported continuous functions on \( X \). 
        \begin{defn} 
            Let $(\mu_N)_{N \in \N} \subset M(X)$. We say that:
            \begin{itemize} 
                \item $(\mu_N)$ {\em converges weak-$^*$} to $\mu$ if for every $f\in C_b(X)$, $\lim_{N\to  \infty} \mu_N(f) = \mu(f)$.
                \item $(\mu_N)$ {\em converges vaguely} to $\mu$ if for every $f\in C_c(X)$, $\lim_{N\to  \infty} \mu_N(f) = \mu(f)$.
            \end{itemize}
        \end{defn}
        It is clear that weak-$^*$ convergence implies vague convergence. Moreover, when $X$ is compact, these two notions are the same. For general $X$, vague convergence is strictly weaker than weak-$^*$ convergence. However, by a result of Prokhorov, vague convergence implies weak-$^*$ under the additional assumption of tightness. Recall that a sequence $(\mu_N)$ is {\em tight} if for every $\varepsilon>0$, there exists $N_\varepsilon \in \N$ and a compact set $K_\varepsilon \subset X$ such that for every $N\geq N_\varepsilon$, we have $\mu_N(K_\varepsilon)>1-\varepsilon$. 
        

    \subsection{Asymptotic Behavior of \( \Omega(n) \)}
    \label{subsec: Number Theory}
        The study of the asymptotic behavior of \( \Omega(n) \) has a rich history in multiplicative number theory, and such questions are often related to fundamental questions about the prime numbers. For instance, let \( \lambda(n) := (-1)^{\Omega(n)} \) denote the Liouville function. Then the Prime Number Theorem is equivalent to the assertion that
        \[
            \lim \AVG \lambda(n) = 0.
        \]
        In other words, asymptotically, \( \Omega(n) \) is even half the time. Further classical results include the Pillai-Selberg Theorem \cite{Pillai1940, Selberg1939} and Erd\H{o}s-Delange Theorem \cite{Erdos1946, Delange1958}. Hardy and Ramanujan \cite[Theorem C']{HR1917} showed that the normal order of \( \Omega(n) \) is \( \log \log n \). 
    
        \begin{theorem}[(Hardy-Ramanujan Theorem)]
        \label{thm: Hardy Ramanujan}
            For \(C > 0\), define \(g_C: \N \to \N\) by 
            \begin{equation}
            \label{eqn: HR}
                g_C(N) 
                = 
                \# \Big\{ n \leq N : |\Omega(n) - \log \log n| > C \sqrt{\log \log N} \Big\}.
            \end{equation}
            Then for all \(  \varepsilon > 0 \), there is some \( C_\varepsilon \geq 1 \) such that 
            \[
                \limsup_{N \to \infty} \frac{g_{C_\varepsilon}(N)}{N} \leq \varepsilon.
            \]
        \end{theorem}
    
        \begin{remark}
            By \cite[Theorem C]{HR1917}, it is equivalent to verify Equation \eqref{eqn: HR} with \( \log \log n\) replaced by \( \log \log N\). 
        \end{remark}
    
        Erd\H{o}s and Kac \cite{EK1940} later generalized \cref{thm: Hardy Ramanujan}, proving that \(\Omega(n)\) becomes normally distributed within such intervals.
    
        \begin{theorem}[(Erd\H{o}s-Kac Theorem)]
        \label{thm: Erdos-Kac}
            For \( A < B \in \R \), 
            \[
                \lim_{N \to \infty} \frac{1}{N} \# \Big\{ n \leq N \,:\, A \leq \frac{\Omega(n) - \log\log N}{\sqrt{\log\log N}} \leq B \Big\}
                = 
                \frac{1}{\sqrt{2\pi}} \int_{A}^B e^{-t^2/2} d t.
            \]
        \end{theorem} 
        
        Thus, the Erd\H{o}s-Kac Theorem states that for large \(N\), the sequence \( (\Omega(n))_{n \leq N} \) becomes roughly normally distributed with mean and variance \(\log \log N\). 
        
    

        In this paper, we convert weighted ergodic averages to those involving weight functions encoding information about the distribution of \(k\)-almost primes. Given \( k \in \N\), let \( \P_k\) denote the set of \(k\)-almost primes, those integers satisfying \( \Omega(n) = k \). Let \( \pi_N(k) \) denote the number of \(k\)-almost primes less than or equal to \( N \).  Then 
        \[
            \pi_N(k) = \# \{ n \leq N \,:\, \Omega(n) = k\} = \sum_{n \in \P_k \cap [N]} 1.
        \]
        Erd\H{o}s \cite[Theorem II]{Erdos1948} provides a uniform estimate for \( \pi_N(k)\) for certain values of \( k \): For \( N \in \N \), define the interval \( I_N \) by 
        \[
            I_N := \Big[ \log \log N - C \sqrt{\log \log N} ~,~ \log \log N + C \sqrt{\log \log N} 
            \Big]   ,
        \]
        where \(C \geq 1\) is the constant guaranteed by \cref{thm: Hardy Ramanujan}. 
        
        \begin{lemma}[\cite{Erdos1948}] 
        \label{lem: Erdos piN Approx}
            The function \( \pi_N(k) \) satisfies  
            \[
                \pi_N(k) = \frac{N}{\log N} \cdot \frac{(\log\log N)^{k-1}}{(k-1)!} (1 + o_{N \to \infty}(1))
            \]
            uniformly for \( k \in I_N \). 
        \end{lemma}

\section{Proof of Theorem \ref{thm:main}}
\label{sec:mainproof}
       
    \begin{proof}[Proof of \cref{thm:main}]
        Let \( (X, T) \) be a measure-preserving system and let \( x \in X\). Suppose \( \mu \in Acc(x) \cap M^e(X,T) \). We want to construct a subsequence \( (N_i)_{i \in \N} \sse \N \) such that 
        \begin{equation}
        \label{eqn: goal}
            \kappa_i := \frac{1}{N_i} \sum_{n \leq N_i} \delta_{T^{\Omega(n)} x}   \text{ is a tight sequence } 
        \end{equation}
        and 
        \begin{equation}
        \label{eqn: goal2}
            \lim_{i \to \infty} \kappa_i(f) = \mu(f) \text{ for all } f \in C_c(X).
        \end{equation}
        Equation \eqref{eqn: goal2} then demonstrates $(\kappa_i)$ converges vaguely to $\mu$, so that by Prokhorov's Theorem, tightness of the sequence implies weak convergence. 
   
        We first perform several reductions. By the separability of \( C_c(X) \), there exists a dense collection \( \{f_j\}_{j =1}^\infty \sse C_c(X) \). By a straightforward density argument, we need only find a subsequence satisfying Equation \eqref{eqn: goal2} for this countable collection. Let $\{B_\ell\}_{\ell \in \N}$ be a family of compact balls such that $\mu(B_\ell)\geq 1-\frac{1}{\ell}$. It is enough to show that for each \( m \in \N\), there exists a sequence  \( (N_i^m)_{i \in \N} \sse \N \) and \( i_m \in \N \) such that for all \( i \geq i_m\), \( j \leq m\),
        \begin{equation}
        \label{eqn: reduction 1}
            \left| \frac{1}{N_i^m} \sum_{n \leq N_i^m} f_j (T^{\Omega(n)}x) - \int f_j \,d\mu \right| < \frac{1}{m},
        \end{equation}
        and for all $i\geq i_m$,  $\ell\leq m$,
        \begin{equation}\label{eq:tight}
         \left| \frac{1}{N_i^m} \sum_{n \leq N_i^m} \delta_{T^{\Omega(n)}x}(B_\ell)\right|\geq 1-\frac{10}{\ell}.
        \end{equation}
        We then define our subsequence \( (N_k) \) as follows: Set \( N_1 = N_{i_1}^1\). For \( k > 1\), set \( N_k = N_{i'_k}^k\), where \( i'_k \geq i_k \) is chosen large enough so that \( N_{i'_k}^k \geq N_{k-1} \). Then for each \( f_j \) and any \( \varepsilon > 0\), there exists \( k_j \geq j\) such that \( \frac{1}{k_j} < \varepsilon\) and for all \( k \geq k_j\),
        \[
            \left| \frac{1}{N_k} \sum_{n \leq N_k} f_j(T^{\Omega(n)}x) - \int f_j \,d\mu \right| < \frac{1}{k_j} < \varepsilon
        \]
        and, moreover, 
        \[
            \left| \frac{1}{N_k} \sum_{n \leq N_k} \delta_{T^{\Omega(n)}x}(B_\ell)\right|\geq 1-\varepsilon
        \]
        for $\ell \geq \frac{1}{10\varepsilon}$, yielding the result. 
        
        Then, to this end, let \( m \in \N \) and consider the collections \( \{f_j\}_{j \leq m} \) and  $\{B_\ell\}_{\ell \leq m}$. For $j\leq m$, let $h_j\in C(B_j) \sse C_c(X) $ be a positive function such that $\mu(h_j) \geq 1-\frac{2}{j}$. 
        Let $f\in \{f_j\}_{j \leq m}\cup \{h_j\}_{j \leq m}$. Regrouping by the value of \( \Omega(n)\), we reformulate the ergodic averages along \( \Omega(n) \) as weighted Birkhoff averages: for \( N \in \N \), 
        \begin{equation}
        \label{eqn: weighted form}
            \AVG f(T^{\Omega(n)}x) =  \sum_{k \leq N} \frac{\pi_N(k)}{N} f(T^k x). 
        \end{equation}
        
        Define \( G_N(k) := e^{-(k-\log \log N)^2/ 2\log \log N}\). Applying \cref{thm: Hardy Ramanujan} and \cref{lem: Erdos piN Approx}, in combination with Stirling's formula, it is straightforward to show that for every $\varepsilon>0$, there exists \( C_\varepsilon > 1\) such that  
        \begin{equation}
        \label{eqn: approximated form}
            \sum_{k \leq N} \frac{\pi_N(k)}{N} f(T^k x) = \frac{1}{\sqrt{2\pi \log \log N}}\sum_{k \in I_N} G_N(k) f(T^k x) + {o}_{\varepsilon \to 0}(1),
        \end{equation}

        \noindent where \( I_N = \log \log N+ [- C_\varepsilon \sqrt{\log \log N} \,,\, C_\varepsilon \sqrt{\log \log N}]\). Note \( C_\varepsilon \) does not depend \(f\) or \(N\). 
        
        We now partition the interval \(I_N\) and aim to use quasi-genericity of the left endpoints of these smaller sub-intervals to obtain the result. However, without the assumption of unique ergodicity, we cannot control the statistical behavior of chosen points. Thus, we introduce a large-measure set on which we do have such control. 
        
        Let \( \varepsilon > 0\).  By Egorov's theorem and ergodicity of \( \mu\), there exists a set \( A_m \sse X \) with \( \mu(A_m) > 1 - \varepsilon \) and an integer \( H_m \in \N \) such that for all \( y \in A_m \), \( f\in \{f_j\}_{j \leq m}\cup \{h_j\}_{j \leq m} \), and \( H \geq H_m \), 
        \[
            \left| \frac{1}{H} \sum_{h \leq H} f(T^h y)  - \int f d\mu \right| < \varepsilon.
        \]
        
        We now alter the set \( A_m \) to be sufficiently nice in order to apply quasi-genericity of \( x\) to \( \mathbbm{1}_{A_m}\): Take a compact set \( K \sse X \) such that \( \mu(K) > 1 - \varepsilon\). Choose \( 0 < \delta < 1\) and cover \( K \) by a finite union of \( \delta-\)balls:
        \[
            K \sse \bigcup_{i=1}^{K_0} B_\delta(y_i)
        \]
        for some \( y_i \in K\). Define 
        \(
            A'_m:= \bigcup_{i \in J} B_{\delta}(y_i),
        \)
        where 
        \[
            J = \left\{ i \leq K_0 \,:\, B_\delta(y_i) \cap A_m \neq \emptyset \right\}. 
        \]
        Since  
        \(
            \mu(K \cap A_m) \geq 1 - 2 \varepsilon,
        \)
        we have 
        \(
            \mu (A_m') 
            \geq 
            \mu(K \cap A_m) 
            \geq 
            1- 2 \varepsilon.
        \)
        Now take \( \delta = \delta(\varepsilon, m, H_0)\) small enough such that for any \( y \in A_m'\),  \( f\in \{f_j\}_{j \leq m}\cup \{h_j\}_{j \leq m} \),
        \begin{equation}\label{eq:H00}
            \left|\frac{1}{H_0} \sum_{h \leq H_0} f (T^n y) - \int f \,d\mu \right|  
            < 
            \varepsilon.
        \end{equation}
        
        Since \( x \) is quasi-generic for \( \mu \), there is a sequence \( ( M_i)\) such that 
        \[
            \lim_{i \to \infty} \frac{1}{M_i} \sum_{n \leq M_i } g(T^n x) 
            = 
            \int g \, d\mu
        \] 
        for all \( g \in C_c(X) \). By the openness of \( A_m' \sse X\), there exists \( (g_k) \sse C_c(X) \) approximating \( \mathbbm{1}_{A_m'}\) from below. Hence 
        \begin{equation}
        \label{eqn: quasigen property}
            \lim_{i \to \infty} \frac{1}{M_i} \sum_{n \leq M_i} \mathbbm{1}_{A_m'}(T^n x) 
            \geq 
            \lim_{k \to \infty} \int g_k \,d\mu 
            = 
            \mu(A_m') 
            > 
            1 - 2 \varepsilon. 
        \end{equation}
        By taking a subsequence, we can assume that $M_{i+1}\geq M_i^2$ for all $i\in \N$. Take \( i_0 \in \N \) such that for \( i \geq i_0 \),
        \begin{equation}
        \label{eqn: using genericity}
            \frac{1}{M_i} \sum_{n \leq M_i}  \mathbbm{1}_{A_m'}(T^n x) > 1 - 3\varepsilon,
        \end{equation}
        and set \( N_1 = M_{i_0}\). Then 
        \[
            \frac{2}{N_1} \sum_{n = N_1/2}^{N_1} \mathbbm{1}_{A_m'}(T^n x) > 1 - 6\varepsilon,
        \]
        else contradicting Equation \eqref{eqn: using genericity}. By pigeonhole, we can then find an integer \( N_1/2 < \ell_1 < N_1 \) such that the interval 
        \(
            I_{\ell_1} := [ \ell_1 - C_\varepsilon \sqrt{\ell_1} \,,\, \ell + C_\varepsilon \sqrt{\ell_1}]
        \)
        satisfies
        \[
            \frac{1}{2C_\varepsilon \sqrt{\ell_1}} \sum_{k \in I_{\ell_1}} \mathbbm{1}_{A_m'} (T^k  x) > 1 - 6 \varepsilon.
        \]
        Now, define  \( N_i = M_{i_0+i-1}\). We can find an integer \( N_i/2 < \ell_i < N_i \) satisfying an analogous condition. Continuing in this way, we obtain an increasing sequence \( (\ell_j)\) such that for all \( j \in \N \),
        \begin{equation}\label{eq:las}
            \frac{1}{2C_\varepsilon \sqrt{\ell_j}} \sum_{k \in I_{\ell_j}} \mathbbm{1}_{A_m'} (T^k  x) > 1 - 6 \varepsilon.
        \end{equation}
        Divide \( I_{\ell_j} \) into \( \frac{2C_\varepsilon \sqrt{\ell_j}}{H_0} \) sub-intervals of length \( H_0 \), call \( J_s \), \( 1 \leq s \leq 2C_\varepsilon \sqrt{\ell_i}/H_0\). 
        
        Let \( u_{s} \) denote the left-endpoint of \( J_s\). Then by Equation \eqref{eq:las} there is an integer \( k_0 \leq H_0 - 1 \) such that the endpoints \(u'_s:=u_s +k_0\) for the shifted intervals \( J'_s=J_s + k_0 \) satisfy
        \begin{equation}
        \label{eqn: left endpoints}
            \frac{H_0}{2C \sqrt{\ell_i}} \sum_{s = 1 }^{2C \sqrt{\ell_i}/H_0} \mathbbm{1}_{A_m'}(T^{u'_s}x) > 1 - 6 \varepsilon.        
        \end{equation}
        Define a subsequence \( (\tilde{N}_i )\) by \( \log \log \tilde{N}_i = \ell_i + k_0\). Then Equation \eqref{eqn: approximated form} implies that for any $f\in \{f_j\}_{j=1}^m\cup \{h_j\}_{j=1}^m$,
        \begin{align*}
            \frac{1}{\tilde{N}_i} \sum_{n=1}^{\tilde{N}_i} f(T^{\Omega(n)}x) 
            &= \frac{1}{\sqrt{2\pi \ell_i}} \sum_{k \in I_{\ell_i}+k_0} G_{\tilde{N}_i}(k) f(T^k x) + {\rm O}(\varepsilon)
            \\
            &= 
            \frac{1}{\sqrt{2 \pi \ell_i}} \sum_{s = 1}^{2C\sqrt{\ell_i}/H_0} \sum_{k = 0}^{H_0-1} G_{\tilde{N}_i}(u'_s) f(T^{k} (T^{u'_s}x)) + {\rm O}(\varepsilon).
        \end{align*}
        Note that for \( k \in J'_s\), we have replaced \( G_{\tilde{N}_i}(k) \) by \(  G_{\tilde{N}_i}(u'_s)\) using uniform continuity of \( e^{-k^2/2}\). Set
        \[
            E_1 := \Big\{ s \leq \frac{2C_\varepsilon \sqrt{\ell_i}}{H_0} \,:\, T^{u'_s}x \in A_m' \Big\}
            \quad \quad \text{and} \quad \quad 
            E_2 := E_1^c \cap [1, 2C_\varepsilon \sqrt{\ell_i}/H_0].
        \]
        We first estimate the sum over \( E_1\). For $u'_s\in E_1$, Equation \eqref{eq:H00} holds for $y=T^{u'_s}x$. Then,
        \begin{align*}
            \left| \frac{H_0}{\sqrt{2 \pi \ell_i}} \sum_{s \in E} G_{\tilde{N}_i}(u'_s) \frac{1}{H_0} \sum_{k = 0}^{H_0-1} f(T^{k}(T^{u'_s}x))-\int_X f d\mu \right| 
            & < 
            \frac{\varepsilon H_0}{\sqrt{2 \pi \ell_i}} \sum_{s \in E} G_{\tilde{N}_i}(u'_s)
            \\
            &\leq \frac{\varepsilon H_0}{\sqrt{2 \pi \ell_i}}\frac{2}{H_0} \sum_{w\in I_{\tilde{N}_i}} G_{\tilde{N}_i}(w)\\
            &\leq 4\varepsilon.
        \end{align*}
        For \( E_2\), by Equation \eqref{eqn: left endpoints}, we get $H_0|E_2|<6\varepsilon \cdot 2C_\varepsilon \sqrt{\ell_i}$, and so 
        \begin{align*}
            \Big| \frac{H_0}{\sqrt{2 \pi \ell_i}} \sum_{s \in E_2} G_{\tilde{N}_i}(u'_s) \frac{1}{H_0} \sum_{k = 0}^{H_0-1} f(T^{k}(T^{u'_s}x)) \Big|
            &
            \leq 
            \frac{H_0 \norm{f}_\infty}{\sqrt{2\pi \ell_i}} \sum_{s \in E_2} G_{\tilde{N}_i}(u'_s)  \\
            & 
            \leq 
            \frac{2H_0 \norm{f}_\infty}{\sqrt{2 \pi \ell_i}}\cdot \sum_{w=1}^{|E_2|} G_{\tilde{N}_i}(\log\log \tilde{N_i}+wH_0)\leq \\&\frac{4H_0 \norm{f}_\infty}{\sqrt{2 \pi \ell_i}}\cdot \frac{1}{H_0}\sum_{w=1}^{H_0|E_2|} G_{\tilde{N}_i}(\log\log \tilde{N_i}+w)
            \\
            &\leq 
            8\varepsilon \|f\|_\infty.
        \end{align*}
        Summarizing, for all  \( f\in \{f_j\}_{j \leq m}\cup \{h_j\}_{j \leq m} \),
        \[
            \left|\frac{1}{\tilde{N}_i} \sum_{n=1}^{\tilde{N}_i} f(T^{\Omega(n)}x)  -\int f \,d\mu \right|= {\rm O}(\varepsilon).
        \]
        Taking $\varepsilon$ much smaller than $1/m$ and $f=f_j$, we obtain Equation \eqref{eqn: reduction 1}, and taking $f=h_j$, we obtain Equation \eqref{eq:tight}, finishing the proof.
         
        \end{proof}

\subsection{Necessity of Ergodicity in \cref{thm:main}}
\label{sec:counterex}

    In this section, we prove \cref{prop:counterex}, demonstrating that \cref{thm:main} does not hold removing the condition of ergodicity. In other words, we find a non-ergodic measure that is visible along standard averages, yet not along \(\Omega\)-averages. In fact, we can take the point \( x \in X\) to be generic for this measure. 

    \begin{proof}[Proof of \cref{prop:counterex}] 
        Let $(X,\sigma)$ be the full shift on the alphabet $\{0,1,2\}$ and let \( \mu = \frac{1}{3}(\delta_{\overline{0}} + \delta_{\overline{1}} + \delta_{\overline{2}})\), where \( \overline{j}\) denotes the sequence with each entry equal to \( j \in \{0,1,2\}\). We construct the point $x \in X$ as follows. Set \( N_0 = 0\) and \(N_1 = 100 \). For \( i \geq 2 \), set \( N_{i} = \floor{N_{i-1}^{4/3}}\). Partition each interval $I_i = [N_i,N_{i+1}]$ into intervals $\{J^i_s\}_{s=1}^{K_i}$ of length $\floor{N_{i+1}^{2/3}}$ (with the last one possibly shorter). We then define the sequence \( \mathbf{x} = (x_n)\) by
        \[
            x_n = 
            \begin{cases}
                0 & n < 0\\
                j & n \in J^i_s, \ s \equiv j \mod 3
            \end{cases}
            .
        \]
        
        Let \( f \in C(X) \). Since each value \( j \in \{0,1,2\}\) appears with asymptotic frequency \( 1/3\) in the definition of \( \mathbf{x}\),
        \begin{align*}
            \left| \frac1{N_{i+1}} \sum_{n \leq N_{i+1}} f(\sigma^n \mathbf{x}) - \int f \, d\mu \right| 
            &= 
            \left| \frac1{N_{i+1}} \sum_{n = N_i}^{N_{i+1}} f(\sigma^n \mathbf{x}) - \frac13 (f(\overline{0}) + f(\overline{1}) + f(\overline{2}))\right| + o(1)\\
            &=
            \frac1{N_{i+1}}\sum_{j = 0}^2  \sum_{n \in J_s, s \equiv j \mod 3}\left| f(\sigma^n \mathbf{x}) - f(\overline{j}) \right| + o(1).
        \end{align*}
        
        Let \( \varepsilon > 0\) and take \( K > 0 \) such that if \( \mathbf{y} \in X \) has its \( K \) entries around 0 equal to \( j \), then \( |f(y) - f(\overline{j})| < \varepsilon\). Notice that for \( n \in J_s^i\), \( \sigma^n \mathbf{x} \) satisfies this for some \( j \in \{0,1,2\} \) except on a set of size at most \( 2K\). Then 
        \begin{align*}
            \frac1{N_{i+1}} \sum_{j = 0}^2  \sum_{n \in J_s, s \equiv j \mod 3}\left| f(\sigma^n \mathbf{x}) - f(\overline{j}) \right| = O(N_{i+1}^{-2/3}) + o_{\varepsilon \to 0}(1).
        \end{align*}
        This implies that $x$ is generic for $\mu$, so that $Acc(x)= \{\mu\}$. On the other hand, using Theorem \ref{thm: Hardy Ramanujan}, we have
        \[
            \frac{1}{N}\sum_{m\leq N} f(T^{\Omega(m)}\mathbf{x}) = \frac1N \sum_{k \in I_N} \pi_N(k) f(T^k \mathbf{x}) + o(1)       
        \]
        for any \( f \in C(X) \). Recall the interval \( I_N \) is centered at \( \log \log N \) with size \( \rm{O}(\sqrt{\log \log N})\). However, for any $M$ large enough, the entries of \( \mathbf{x} \) in a window of shape 
        \[
            [ M - M^{1/2+1/100}, M + M^{1/2+1/100}]
        \]
        will have at most two of the symbols $\{0,1,2\}$ (for most times, such a window will contain only one symbol, though at transitions, it may contain two). Therefore,
        $\frac{1}{M}\sum_{m\leq M}\delta_{T^{\Omega(m)}x}$ cannot converge to $\mu$. This finishes the proof.
    \end{proof}
  
    \begin{remark}
        This construction can be adapted to any sequence $(a_n)\subset \N$ for which there exists a sequence of intervals $\{I_N\}$ satisfying $|I_N| = o(N)$ and \( |\{n\leq N\;:\; a_n\notin I_N\}| = o(N).\)
    \end{remark}

\section{Background on Horocycle flows}
    Let $G=PSL(2,\R)$ and $\Gamma \subset G$ a non-uniform lattice. Set $X=\Gamma \backslash G$ and let \( \mu_X \) denote the Haar measure on \( X \). Let $h_t=\begin{pmatrix}1&t\\0&1 \end{pmatrix}$ be the horocycle flow acting by right multiplication: 
    \begin{equation}\label{eq:hor}
        h_t(\Gamma x)=\Gamma x \cdot h_t.
    \end{equation}
    Let
    \[
        g_s=\begin{pmatrix}e^{s/2}&0\\0&e^{-s/2} \end{pmatrix}\;\; \text{ and }\;\; v_t=\begin{pmatrix}1&0\\t&1 \end{pmatrix}
    \]
    denote the geodesic and complementary horocycle flow, respectively. The flows $(g_s)$ and $(v_t)$ act on $X$ analogously to Equation \eqref{eq:hor}. Recall the following renormalization relations: for all \(s,t\in \R\),
    \begin{equation}\label{eq:norm}
    h_t \cdot g_s=g_s\cdot  h_{e^{-s}t}\;\;\;\text{ and }\;\;\; v_t\cdot g_s=g_s\cdot v_{e^{s}t}.
    \end{equation} 
    It is known that for some $k\geq 1$, the space $X$ has $k$ inequivalent cusps and for each cusp $C_i$, there exists a one parameter family of periodic measures corresponding to $C_i$:
    
    \begin{lemma}[(Lemma 11.29, \cite{EW})]
    \label{lem:perppoint}
        Let $p$ be a periodic point. Then $g_t(p)$ diverges to exactly one of the cusps. Moreover any periodic orbit associated to this cusp is of the form $g_{t}(p)$ for some $t\in \R$.
    \end{lemma}

    \begin{remark}
        In the remaining sections, the letter \( k \) will be reserved to denote the number of inequivalent cusps of \( X \).
    \end{remark}
    
    Let $e_i$ be a periodic point of period $1$ corresponding to the $i$-th cusp. It follows from \cref{lem:perppoint} that all other periodic points associated to this cusp are of the form $g_s(e_i)$ for some $s\in \R$. For \( i \leq k \) and \( s \in \R \), denote by $\nu^i_s$ the unique $(h_t)$-invariant measure supported on the closed $(h_t)$ orbit of $g_s(e_i)$. It follows that any periodic invariant measure for $(h_t)$ is one of the measures $\{\nu_s^i\}_{i\leq k, s\in \R}$. Then denote by $\{\nu^i_s\}_{s\in \R, i\leq k}$ the set of periodic measures for the horocycle flow. By work of Dani and Smillie, \cite{D-S}, all ergodic \((h_t)\)-invariant measures are given by either $\mu_X$ or one of the $\{\nu^{i}_s\}$. Hence for every $x\in X$, $Acc(x)$ is a singleton, equal to either $\mu_X$ or $\nu^{i}_s$ for some \( i \leq k, s \in \R \).
    
    Let $d_G$ denote the left invariant metric on $G$ and let $d_X(\Gamma g_1,\Gamma g_2)=\inf_{\gamma \in \Gamma} d_G(\gamma g_1,g_2)$. We denote $dist(x,e)= d_X(x,e)$.  The following quantitative equidistribution result was proven by Str\"{o}mbergsson \cite[Theorem 1.1]{Strombergsson} (see also \cite[Theorem 5.14]{Flaminio-Forni}):
    
    \begin{theorem}
    \label{thm:SFF} 
        There exists $\alpha\in (0,1)$ such that for every $f\in C_c^4(X)$ and $x\in X$, 
        $$
            \frac{1}{T}\int_0^T f(h_t(x))dt=\mu_{X}(f)+ {\rm O}(\|f\|_{W_4})r^{-\alpha},
        $$
        where $\|f\|_{W_4}$ denotes the Sobolev norm and $r=r(x,T)=T\cdot e^{-dist(g_{\log T}(x),e)}$.
    \end{theorem}
    
    We have an analogous result for the flow $(v_t)$ that we state for future reference:
    \begin{remark}
    \label{rem:SFF} 
        There exists $\alpha\in (0,1)$ such that for every $f\in C_c^4(X)$ and $x\in X$, 
        \[
            \frac{1}{T}\int_0^T f(v_t(x))dt=\mu_{X}(f)+ {\rm O}(\|f\|_{W_4})\bar{r}^{-\alpha}
        \]
        where $\bar{r}=\bar{r}(x,T)=T\cdot e^{-dist(g_{-\log T}(x),e)}$.
    \end{remark}
    
    We additionally require a discrete version of the above theorem, i.e. for the time one map of the horocycle flow. The following theorem appears in \cite[Theorem 1.4]{Stre}  (see also \cite[Theorem 1.2]{Zhe}). 
    
    \begin{theorem}
    \label{thm:discSFF} 
        There exists $\alpha\in (0,1)$ such that for every $f\in C_c^4(X)$ and $x\in X$, 
        $$
            \frac{1}{T}\sum_{n=0}^T f(h_n(x)) =\mu_{X}(f)+ {\rm O}(\|f\|_{C^4})r^{-\alpha}
        $$
        where $r=r(x,T)=T\cdot e^{-dist(g_{\log T}(x),e)}$.
    \end{theorem}
    
    Note that if \( x \in X \) is periodic, the error $r(x,T)$ should be large. Indeed, the following result \cite[Lemma 1.3]{Stre} relates the size of $r(x,T)$ with closeness to a periodic point:
    \begin{lemma}\label{lem:str}
        Let $x\in X$ and $\delta>0$. Let $T\geq 0$ and $K\leq T$. There is an interval $I_0\subset [0,T]$ of size $|I_0|<\delta^{-1}K^2$ such that: 
        \begin{itemize}
            \item For all $t_0\in [0,T] \setminus I_0$, there is a segment $\{h_t(p) \,:\, t\leq K\}$ of a closed horocycle approximating  $\{h_{t_0+t}(x) \,:\, 0\leq t\leq K\}$ of order $\delta$, in the sense that 
            $$
                \sup_{0\leq t\leq K} \;\; d_X(h_{t_0+t}x,h_tp)\leq \delta;
            $$
            \item The period $P= P(t_0,p)$ of this closed horocycle is at most 
            \[
                P \ll  r(x,T) = T e^{-dist(g_{\log T}x,e))}.
            \]
            Moreover, one can ensure $P \gg \eta^2 r$ for some $\eta>0$ by weakening the bound on $I_0$ to $|I_0|\leq \max (\delta^{-1}K^2, \eta T)$ (here $\ll$ depends only on $\Gamma$).
        \end{itemize}
    \end{lemma}

    \begin{remark}
        Here, $\ll$ denotes Vinogradov's notation and will be used frequently in the following sections to simplify certain asymptotic expressions.
    \end{remark}

\section{Distribution of orbits of horocycle flows}
In this section we gather our main technical results concerning distribution of horocycle orbits and their closeness to periodic points. 
    \begin{lemma}
    \label{lem:approx} 
        Fix a compact set $K\subset X$. For $z\in K$, set 
        \[
            \tilde{B}:=\{h_ag_bv_c(z)\;:\; a\in [0,1], |b|<\eta, |c|<\eta'\}.
        \]
        Then for any $x\in X$,
        $$
        \frac{1}{T}\int_0^T \mathbbm{1}_{\tilde{B}}(h_t(x))dt=\mu_{X}(\tilde{B})+ {\rm O}((\eta\eta')^{-5})r^{-\alpha},
        $$
        where $r=r(x,T)$ is as in Theorem \ref{thm:SFF}.
    \end{lemma}
    \begin{proof} 
        The proof is a straightforward consequence of Theorem \ref{thm:SFF} and approximating $\mathbbm{1}_{\tilde{B}}$ by a smooth function with control on the $C^4$ norms in the $(g_t)$ and $(v_t)$ directions.
    \end{proof}

    \begin{lemma}
    \label{lem:perp} 
        Let $p\in X$ be an \(h_t\)-periodic point of period $P$. There exists constants $C_P>10$ and $T_P>0$ such that for every $T\geq T_P$, there exists $c\in [C_PT^{-1/2},2C_P T^{-1/2}]$ satisfying 
        $$
        g_{\log(C_PT)}(v_c p)\in B_X(e,1):=\{x\in X\;:\; d_X(x,e)\leq 1\}.
        $$
    \end{lemma}
    
    \begin{proof} 
        Since $p$ is a fixed periodic point, write $C=C_P$. Note that by Equation \eqref{eq:norm}, 
        $$
        g_{\log (CT)}\Big(\{v_{c}(p)\;:\; c\in  [CT^{-1/2}, 2CT^{-1/2}]\}\Big)= \left\{v_{r}(g_{\log(CT)}p)\;:\; r\in  [C^2T^{1/2}, 2C^2T^{1/2}]\right\}.
        $$
        Applying \cref{rem:SFF} with $B=B_{X}(e,1)$, there exists a constant \(C'\) depending only on $B$ such that 
        $$
        \left|\frac{1}{2C^2{T^{1/2}}}\int_{0}^{2C^2T^{1/2}}\mathbbm{1}_{B}(v_r (g_{\log(CT)}p)) dr-\mu_{X}(B)\right|\leq C'\bar{r}^{-\alpha},
        $$
        where 
        \[
            \bar{r}=2C^2T^{1/2} \cdot e^{-dist(g_{-\log (2C^2T^{1/2})}(g_{\log(CT)}p),e)}=2C^2T^{1/2} \cdot e^{-dist(g_{\log( T^{1/2}/2C))}p,e)}.
        \]
        
        Since $p$ is periodic with period $P$, it follows that $p=g_{\log P}(h_z(e_i))$ for some $i\leq k$ and $|z|\leq 1$. We then obtain 
        \begin{align*}
            dist(g_{\log( T^{1/2}/2C))}(p),e) &= dist(g_{\log(P T^{1/2}/2C)}(h_z (e_i)),e)\\
            &\leq 
            d_X(g_{\log(P T^{1/2}/2C)}(h_z(e_i)),h_z(e_i))+dist(h_z(e_i),e)\\
            &\leq 
             \log(PT^{1/2}/C)).
        \end{align*}
        Therefore $\bar{r}\geq 2C^3/P$. On the other hand, by analogous reasoning, 
        $$
            \left|\frac{1}{C^2{T^{1/2}}}\int_{0}^{C^2T^{1/2}}\mathbbm{1}_{B}(v_r (g_{\log(CT)}p)) dr-\mu_{X}(B)\right|\leq C'\tilde{r}^{-\alpha},
        $$
        where
        $$
            \tilde{r}=C^2T^{1/2} \cdot e^{-dist(g_{-\log (C^2T^{1/2})}(g_{\log(CT)}p),e)}=C^2T^{1/2}\cdot e^{-dist(g_{\log PT^{1/2}/C)}(h_ze_i),e)}\geq \frac{C^3}{4P}.
        $$ 
        Taking both of inequalities into account, for $C=C(P)>0$ is large enough, 
        $$
        \int_{C^2T^{1/2}}^{2C^2T^{1/2}}\mathbbm{1}_{B}(v_r (g_{\log(CT)}p))dr >0.
        $$
        This finishes the proof.
    \end{proof}

    The next lemma provides quantitative control on the distance from the horocycle orbit of points with bounded geodesic orbit to periodic points. 
    
    \begin{lemma}
    \label{lem:low}
        Let $p\in X$ be an \((h_t)\)-periodic point. Let $x\in X$ be such that $\{g_{s}x\}_{s\in \R}$ is bounded in $X$. Then there exists a constant $d_{x,p}>0$ such that if \[
        h_Tx=\gamma p \begin{pmatrix}a&0\\c&a^{-1}\end{pmatrix}
        \quad \text{ with } \ |a-1|< 1, \ |c|<1
        \]
        for some sufficiently large  $T>0$ and $\gamma\in \Gamma$, then 
        $|c|>d_{x,p}T^{-1/2}$.
    \end{lemma}
    
    \begin{proof} 
        Let \( p \) be an \(h_t\)-periodic point. For $d>0$, set 
        \[
            V_{T}(d,p)
            =
            \left\{\Gamma p\begin{pmatrix}a&0\\c&a^{-1} \end{pmatrix}\;:\; |a-1|<1 \text{ and }|c|\leq d T^{-1/2}\right\}.
        \]
        Let \( x \in X \) be a point with bounded geodesic orbit in \( X \). Suppose for some \( T > 0\), \( \gamma \in \Gamma\), 
        \begin{equation}
        \label{eqn: approx to per}
            h_T x = \gamma p \begin{pmatrix}a&0\\c&a^{-1}\end{pmatrix}            
        \end{equation}
        with \( |a-1|<1\), and \(|c|< 1\). We want to show that if $d$ is small enough, depending on \(x\) and \(p\), then $h_T x\notin V_T(d,p)$ for all $T>0$ sufficiently large. Suppose \( h_T x \in V_T(d, p)\). Then \( |c| < dT^{-1/2}\). Hence, composing both sides of Equation \eqref{eqn: approx to per} with \( g_{\log T}\) and applying renormalization, 
        \[
            h_1(g_{\log T}(x))\in \Big\{ v_{c'} (g_{\log T+2 \log a}(\Gamma p)) \;:\; |c'|\leq ad T^{1/2}\Big\}.
        \]
        By assumption on \( x \), there exists \( K > 0\) such that \( d_X(g_{\log T}x, e) < K\) so that, possibly enlarging \( K \) to \( d_X(h_1, e)\), we have
        \[
            d_X (h_1 (g_{\log T}x), e) < 2K. 
        \]
        On the other hand, by the triangle inequality, 
        \begin{align*}
            d_X(v_{c'}(g_{\log T+2 \log a}(p)),e) 
            &\geq 
            d_X( g_{\log T+2 \log a}(p), e) - d_X(v_{c'}, e)
            \\
            &\geq d_X(g_{\log T}(p),e) - d_X(g_{2 \log a}, e)- d_X(v_{c'}, e)
            \\
            &\geq d_X( g_{\log T}, e) - d_X(p,e) - |2\log a| - d_X(v_{c'}, e)
            \\
            &\geq \log T - (d_X(p,e) +1) - (2 \log d + \log T)
        \end{align*}
        for sufficiently large \( T\). Hence 
        \[
            d_X(v_{c'}(g_{\log T+2 \log a}(p)),e) \geq -2 \log d - (d_X(p,e) +1) > 2K
        \]
        for \( d = d_{x,p}> 0\) small enough. Hence \( h_T x \notin V_T(d_{x,p}, p)\) for sufficiently large \(T\).

    \end{proof}

    \begin{lemma}
    \label{lem:ttadd} 
        Let \( y \in X\). Assume there is a periodic point $p\in X$ and $\eta,R>0$ satisfying
        $$
            \sup_{0\leq t\leq \frac{1}{2}\sqrt{\eta^2\delta R}} \;\; d_G(h_{t}y,h_{t}p)\leq \delta
        $$
        for some \( 0 < \delta < 1/2 \). Then 
        \[
            y= h_v (p) \begin{pmatrix}a&0\\c&a^{-1}\end{pmatrix},
        \]
        where $|a-1|\ll (\eta^2\delta R)^{-1/2}$, $|c|\ll \eta^{-2}R^{-1}$, and $|v|<1$.
    \end{lemma}
    
    \begin{proof}
        Let \( y \in X \) satisfy the assumptions in the statement of the lemma. Then, for every $0\leq t\leq \frac{1}{2}\sqrt{\delta R}$, the matrix $h_{-t}\cdot  p^{-1}\cdot y \cdot h_{t}$ is $\delta$ close to the identity matrix. Write 
        \[
            p^{-1} \cdot y = \begin{pmatrix}1&d\\0&1\end{pmatrix} \begin{pmatrix}a&0\\c&a^{-1}\end{pmatrix}
        \]
        for some \(a, c, d \in \R\). By a direct computation,  
        $$
        h_{-t}\cdot  p^{-1} \cdot y \cdot h_{t}=\begin{pmatrix}a+c(d-t)&-ct^2+(a-a^{-1})t + d(a^{-1}+ct)\\c&a^{-1}-ct\end{pmatrix}.
        $$
        By assumptions, $|-ct^2+(a-a^{-1})t| \ll \delta$ for $0\leq t\leq \frac{1}{2}\sqrt{\eta^2\delta R}$. Let \( P(t) = -ct^2+(a-a^{-1})t\) and \( S = \frac{1}{2}\sqrt{\eta^2\delta R}\). Notice that 
        \[
            P(S) - 2 P\left( S/2 \right) = - \frac12 c S^2 ,
        \]
        so that $|c| \ll \delta/S^2 \ll \eta^{-2}R^{-1}$. Additionally, 
        \[
            |(a+a^{-1})S| = |P(S) + c S^2| \ll \delta + |c|S^2 \ll \delta,
        \]
        so that $|a-a^{-1}| \ll \delta/S \ll (\eta^2 \delta R)^{-1/2}$. Hence $|a-1| \ll (\eta^2 \delta R)^{-1/2}$ as well. Finally, applying the assumption for \( t = 0\), we obtain \( |d| < |a| \delta < 1 \). Taking \( v = d\), we are done.  
    \end{proof}

    \begin{lemma}
    \label{lem:tadd} 
        Let \( \varepsilon > 0 \) and $p$ an \((h_t)\)-periodic point. Suppose
        $$y=\gamma p\begin{pmatrix}1&0\\c&1\end{pmatrix}$$
        for some $|c|\ll R^{-1}<\varepsilon^3$. Then for any interval $J=[u,v]$ satisfying $|J|\leq \varepsilon^3 R$ and $|1+ct|>\varepsilon$ for all $t\in J$, we have 
        $$
        \sup_{t\in J}\;\;d_X\Big(h_{\frac{t}{1+ct} (1+cu)^2}(p_u), h_t(y)\Big) \ll \varepsilon,
        $$
        where $p_u=g_{-2\log|1+cu|}(p)$.
    \end{lemma}
    
    \begin{proof} 
        Let $\bar{c}(t)=\frac{c}{1+ct}$. By a direct computation,
        \begin{align*}
            h_t(y) &=\gamma p\begin{pmatrix}1&0\\c&1\end{pmatrix}\begin{pmatrix}1&t\\0&1\end{pmatrix} 
            \\
            &=
            \gamma p \begin{pmatrix}1&\frac{t}{1+ct}\\0&1\end{pmatrix}\begin{pmatrix}(1+ct)^{-1}&0\\0&1+ct\end{pmatrix} \begin{pmatrix}1&0\\ \bar{c}(t)&1\end{pmatrix} 
            \\
            &= 
            \gamma p\begin{pmatrix}(1+cu)^{-1}&0\\0&1+cu\end{pmatrix}\begin{pmatrix}1&\frac{t}{1+ct}(1+cu)^2\\0&1\end{pmatrix}\begin{pmatrix}\frac{1+cu}{1+ct}&0\\0&\frac{1+ct}{1+cu}\end{pmatrix} \begin{pmatrix}1&0\\ \bar{c}(t)&1\end{pmatrix}
            \\
            &=
            h_{\frac{t}{1+ct} (1+cu)^2}(p_u) \begin{pmatrix}\frac{1+cu}{1+ct}&0\\0&\frac{1+ct}{1+cu}\end{pmatrix} \begin{pmatrix}1&0\\ \bar{c}(t)&1\end{pmatrix}.
        \end{align*}
        Note that $g$ and $(-id)g$ are identified and so we have taken $1+cu$ instead $|1+cu|$ in the definition of \( p_u \). Finally, by our assumptions, $|\bar{c}(t)|\ll \varepsilon^3 \varepsilon^{-1}<\varepsilon^2$ and 
        \[
            \left|\frac{1+cu}{1+ct}-1\right|= |c|\frac{|u-t|}{|1+ct|}\ll R^{-1}\varepsilon^3 R \varepsilon^{-1}= \varepsilon^2.
        \]
        This implies the lemma.
    \end{proof}

    \begin{lemma}\label{lem:ntc}
        Let $x$ be non-periodic for \( (h_t) \) and let $C > 0$. Then for any increasing sequence $(T_i)$ diverging to $\infty$ and any sequence of periodic points $(p_i)$ with periods in $[C^{-1},C]$, if 
        $$
            h_{T_i}x
            =
            \gamma p_i\begin{pmatrix}1&0\\c_i&1\end{pmatrix}
        $$
        for some numbers $(c_i)$, then $|c_i|T_i\to \infty$.
    \end{lemma}
    \begin{proof} 
        Assume otherwise that there exists \( M \in \R \) such that, up to taking a subsequence, $|c_i|T_i\leq M$ for all \( i \). Let \( J_i \subset [0,T_i] \) be equal to either \( [\frac12 T_i, \frac35 T_i ] \) or \( [\frac45 T_i, T_i ]\) so that \( |1 - c_i t| > \frac1{100} \) for all \( t \in J_i \). Notice that at least one of the intervals must satisfy this. By similar computations as in \cref{lem:tadd}, we have
        \[
            h_{-t+ T_i}x = \gamma p \begin{pmatrix}1&\frac{-t}{1-c_i t}\\0&1\end{pmatrix}\begin{pmatrix}(1-c_i t)^{-1}&0\\0&1-c_i t\end{pmatrix} \begin{pmatrix}1&0\\ \bar{c}_i(t)&1\end{pmatrix} ,
        \]
        where \( \bar{c}_i(t) = \frac{c_i}{1-c_i t}\). Let \( {\rm Orb}(p') \) denote the closed \( h_t\)-orbit of a periodic point \( p' \) and define \(K_M \) to be the closure of the set
        \[
            \bigcup_{|u|<1} v_u \left( \bigcup_{\min\{\frac{1}{100}, \frac1M\} < s< \max\{100, M\}}g_{2\log s} \left( \bigcup_{p'}{\rm Orb}(p')\right)\right),
        \]
        where \( p' \) ranges over periodic points of period \( [C^{-1}, C]\). Note that $K_M$ is compact. Moreover, by our assumption on \( |1-c_it|\), we have that \( h_{-t+T_i} x \in K_M \) for all \( t \in J_i\). 
        
        However, since $x$ is non-periodic, it is generic for the Haar measure, and so it follows by the ergodic theorem that for a fixed positive proportion of $t\in J_i$, $h_t x \notin K_M$. This contradiction finishes the proof.
    \end{proof}

    We now show that the sheared orbits of periodic points appearing in the conclusion of \cref{lem:tadd} have the desired statistical behavior over the intervals \( J = [u,v]\), subject to additional constraints: 
    
    \begin{lemma}
    \label{lem:ns} 
        Let $p$ be a periodic point. Let $c,\varepsilon>0$ and let  $I=[u,v]$ be an interval such that $\varepsilon^{-1}>|1+ct|>\varepsilon$ on $I$ and such that $|cu|\ll_\varepsilon 1, |c(v-u)|\ll_\varepsilon 1$, and $|c||I|^2\to \infty$. Then for all \(\psi \in C_c(X)\),
        $$
            \frac{1}{|I|}\sum_{t \in I}\psi\left(h_{\frac{t}{1+ct}(1+cu)^2}p\right)=\nu_{p}(\psi)+o_{|I|\to \infty}(1).
        $$
    \end{lemma}
    
    To prove this, we will use the following second derivative test for exponential sums:

    \begin{lemma}[\cite{Rob}, Theorem 1] 
    \label{lem:sdt}
        Let $\beta\geq 1$ be a real number. There exists a constant $C_\beta > 0$ such that for any interval $I$, real number $\lambda>0$, and twice differentiable function $f:I\to \R$ satisfying 
        $$
            \lambda \leq \left|f''(x)\right| \leq \beta\lambda
        $$
        for all \( x \in I \), one has 
        $$
            \left| \sum_{t \in I}e^{2\pi i f(t)}\right|\leq C_\beta \Big( |I|\lambda^{1/2}+ \lambda^{-1/2}\Big).
        $$
    \end{lemma}

    \begin{proof}[Proof of \cref{lem:ns}] 
        Since $p$ is a fixed \(h_t-\)periodic point, it follows that the horocycle flow restricted to the closed orbit of $p$ is conjugated (via a smooth map with bounded derivatives) to a linear flow on the circle by angle $\alpha=\frac{1}{per(p)}$. Then $\psi$ can be identified with a function on the circle with mean $\nu_{p}(\psi)$. Using Fourier analysis, it is therefore enough to show that for each $\ell\in \Z\setminus 0$,
        $$
            \frac{1}{|I|}\sum_{t \in I}e^{2\pi i \ell \frac{t}{1+ct}(1+cu)^2\alpha}=o_{|I|\to \infty}(1).
        $$
        Take $f(x)=\ell \frac{x}{1+cx}(1+cu)^2\alpha$. Note that $f''(x)= \ell\alpha(1+cu)^2 \frac{-2c}{(1+cx)^3}$, so that 
        $$
            2 \ell \alpha(1+cu)^2c\varepsilon^3 \leq  |f''(x)| \leq 2 \ell \alpha(1+cu)^2c\varepsilon^{-3}.
        $$ 
        Applying \cref{lem:sdt} with $\lambda= 2 \ell \alpha(1+cu)^2c\varepsilon^{3}$ and $\beta=\varepsilon^{-6}$, 
        $$
            \left|\sum_{t \in I}e^{2\pi i \ell \frac{t}{1+ct}(1+cu)^2\alpha}\right|
            \ll _\varepsilon |I|\lambda^{1/2}+ \lambda^{-1/2}.
        $$
        By assumptions on $I$ and $c$, we obtain $c \ll_\varepsilon |I|^{-1}$, so that $|I|\lambda^{1/2}=o(|I|)$ and $\lambda^{-1/2}=o(|I|)$. This finishes the proof.
        
    \end{proof}

\subsection{Main approximation lemma} 
In this section, we prove the main approximation lemma, which relates the behavior of $\Omega$-ergodic averages and periodic measures.

\begin{lemma}
\label{lem:fini}
    For \( N \in \N \), set $T = \log \log N + z\sqrt{\log \log N}$, where $z \ll 1$. Assume that there exists a periodic point $p=g_{s}(e_i)$, where \( s \in \R \) and \( i \leq k \), such that
    $$
        h_{T}(x)
        = 
        \gamma p \begin{pmatrix} a & 0\\ c & a^{-1} \end{pmatrix}\cdot g_{\xi'} \cdot h_\xi
    $$
    for some $|a-1|<1, |c| \ll T^{-1/2},$ $|c|T \to \infty$, and $|\xi'|, |\xi| \ll e^{T^{-100}}$. Then for all $\psi\in C_c(X)$,
    $$
    \left|\frac{1}{N}\sum_{n\leq N} \delta_{h_{\Omega(n)}x}(\psi) - \frac{1}{\sqrt{2\pi}}\int_{-\infty}^{\infty}e^{-\frac{r^2}{2}}\nu^{i}_{s + 2 \log a-2\log|1+c'(r-z)|}(\psi) dr\right|
    =
    {o}_{N \to \infty}(1),
    $$ 
    where \( c ' = ac \sqrt{\log \log N}\).
\end{lemma}

\begin{proof}  
    Let \( N \in \N \) and set \( T = \log \log N + z \sqrt{\log \log N} \) with \( z \ll 1\). First, note that by the bounds on $\xi, \xi'$, for $n \in [T/2,2T]$, 
    \[
        d_G\Big(h_n\cdot (g_{\xi'} \cdot h_\xi)\cdot h_{-n},e\Big)
        <
        T^{-2}.
    \]
    Then, without loss of generality, we can assume that $\xi=\xi'=0$. Given $\varepsilon>0$, let $C_\varepsilon>0$ be that guaranteed by \cref{thm: Hardy Ramanujan} and set 
    \[
        I_N(C_\varepsilon) = \log \log N + [  - C_\varepsilon \sqrt{\log \log N} \,,\,  C_\varepsilon \sqrt{\log \log N}].
    \]
    Note that $T\in I_N(C_\varepsilon)$ for $\varepsilon$ small enough. For $j\in \Z$, define 
    \[
        K_j
        =
        T+\left[j \varepsilon^4 \sqrt{\log \log N},(j+1) \varepsilon^4 \sqrt{\log \log N}\right]
    \] 
    and consider intervals $\{K_j\}_{j=-U}^V$ covering $I_N(C_\varepsilon)$. Then by \cref{thm: Hardy Ramanujan} and \cref{lem: Erdos piN Approx}, for any fixed $\psi\in C^1_c(X)$, we have
    $$
        \left|\frac{1}{N}\sum_{n\leq N} \psi(h_{\Omega(n)}x)- \frac{2C_\varepsilon}{\sqrt{2\pi}(V+U)}\sum_{j=-U}^V \frac{G_N(j)}{|K_j|}\sum_{t \in K_j}\psi(h_{t}x)\right| \ll \varepsilon,
    $$
    where $G_N(j)=G_N(k_j)$ and $k_j$ is the right endpoint of $K_j$. Consider those $j\in \{-U,\dots,V\}$ for which $|1+ac(t-T)|<\varepsilon$ for some $t\in K_j$. Since the function $1+ac(t-T)$ is linear, it follows that for $\varepsilon$ small enough, the union of all such $\{K_j\}$ has a $\ll \varepsilon$ contribution to the sum above, and so they can be discarded. Then, without loss, assume that $|1+act|>\varepsilon$ for all $t\in K_j$ and \( j \in  \{-U,\dots,V\}\). 

    Now, set \( p_a = g_{2 \log a}(p)\). Then, by our assumption, 
    \[
        h_{T}x
        = 
        \gamma p \begin{pmatrix} a & 0\\ c & a^{-1} \end{pmatrix} = \gamma p_a \begin{pmatrix} 1 & 0\\ ac & 1 \end{pmatrix}.
    \]
    Let \( u_j = j \varepsilon^4 \sqrt{\log \log N} \) denote the left hand endpoint of the interval \( K_j - T \). Then for each \( j \in \{-U, \dots, V\} \), applying Lemma \ref{lem:tadd} with $y=h_{T}x$, $p=p_a$, $R = T^{1/2} \ll \sqrt{\log\log N} $, and $J= K_j - T$, implies that
    \begin{align*}
        \frac{1}{|K_j|}\sum_{t \in K_j}\psi(h_{t}x)&=\frac{1}{|K_j|}\sum_{t\in K_j}\psi(h_{t-T}h_{T}x)\\
        &=
        \frac{1}{u_{j+1}-u_j}\sum_{t\in [u_j,u_{j+1}]}\psi(h_{t}(h_{T}x))\\
        &= 
        \frac{1}{u_{j+1}-u_j}\sum_{t \in [u_j,u_{j+1}]}\psi(h_{\frac{t}{1+tac}(1+acu_j)^2}p_{u_j})+{o}_{\varepsilon \to 0}(1). 
    \end{align*}
    Then the periodic measure corresponding to $p_{u_j}$ is $\nu^i_{s+2\log a-2\log|1+acu_j|}$. Note that by our assumptions, $|1+act|>\varepsilon$ for $t\in [u_j,u_{j+1}]$. Additionally, 
    \[
        |act|\ll R^{-1} (j+1) \varepsilon^4 \sqrt{\log\log N} \ll_\varepsilon 1.
    \]
    Analogously, $|ac u_j|\ll_\varepsilon 1$, $|ac(u_{j+1}-u_j)|\ll_\varepsilon 1$, and $|c||u_{j+1}-u_j|^2\to \infty$. Hence, by \cref{lem:ns}, it follows that 
    $$
        \frac{1}{u_{j+1}-u_j}\sum_{t\in [u_j,u_{j+1}]}\psi \left(h_{\frac{t}{1+act}(1+acu_j)^2}p_{u_j} \right)=\nu^i_{s+2\log a-2\log|1+acu_j|}(\psi)+ {o}_{\varepsilon \to 0}(1).
    $$

    Let $c'= ac\sqrt{\log\log N}$. Then $|c'|\ll R^{-1}\sqrt{\log\log N} \ll 1$ and we have
    
    $$
        \left|\frac{1}{N}\sum_{n\leq N} \psi(h_{\Omega(n)}x)- \frac{2C_\varepsilon}{\sqrt{2\pi}(V+U)}\sum_{j=-U}^V G_N(j) \nu_{s+2\log a -2\log|1+c'\varepsilon^4j|}(\psi)\right| =  {o}_{\varepsilon \to 0}(1).
    $$
    
    Note that $G_N(j)=G_N(k_j)=e^{-\frac{(z+j\varepsilon^4)^2}{2}}$. Hence, by the definition of Riemann integral, $\frac{2C_\varepsilon}{\sqrt{2\pi}(V+U)}\sum_{j=-U}^V G_N(j) \nu_{s+2\log a-2\log|1+c'\varepsilon^4j|}(\psi)$ is $\ll \varepsilon$ distant from 
    $$\Big(\frac{1}{\sqrt{2\pi}}\int_{-\infty}^{\infty}e^{-\frac{(z+r)^2}{2}}\nu^{i}_{s+2\log a-2\log|1+c'r|} dr\Big)(\psi).$$
    Using the change of variables $r'=r+z$, we obtain  
    $$
    \Big(\frac{1}{\sqrt{2\pi}}\int_{-\infty}^{\infty}e^{-\frac{r^2}{2}}\nu^{i}_{s+2\log a - 2\log|1+c'(r-z)|} dr\Big)(\psi).
    $$
    This completes the proof.
\end{proof}

\section{Proof of Theorem \ref{thm:horo} and Theorem \ref{thm:horoboun}}

In this section, we prove Theorems \ref{thm:horo} and \ref{thm:horoboun}. The proofs are split in three subsections. In the first, we prove Equation \eqref{eq:subs} holds. In the second, we prove Theorem \ref{thm:horoboun}. Finally, we prove Equation \eqref{eq:subs'} holds in the third subsection. We follow this order as parts of the proof of Theorem \ref{thm:horoboun} rely significantly on results proved while showing Equation \eqref{eq:subs}. 

    \subsection{Proof of Equation \eqref{eq:subs}}
    \label{subsec:eqn1 proof}
        Fix a non-periodic $x\in X$ and let \( \kappa \in Acc^\Omega(x)\). Let $(N_i) \sse \N$ be a sequence for which 
        \[
        \frac{1}{N_i}\sum_{n\leq N_i} \delta_{h_{\Omega(n)}x}\to \kappa.
        \]
        Note that $\kappa$ is a measure on the compactification $X_\infty$ of $X$, as there might be escape of mass. Assume $\kappa \neq \mu_X$, else we are done. We consider the following two cases (at least one must hold): \\
        
        \noindent \textbf{Case I.} There exist subsequences $(N_{i_m}) \sse (N_i)$ and $(r_m) \sse \R$ tending to infinity such that $h_t(x) \notin B(e,r_m)$ for all \( t \in \R \) satisfying
        $$
        |t- \log \log N_{i_m}|<r_m\sqrt{\log \log N_{i_m}} .
        $$
        Let $\varepsilon>0$ and let $C=C_\varepsilon>0$ be the constant guaranteed by Theorem \ref{thm: Hardy Ramanujan}. Then for $m$ large enough, there is a set $A_m$ of cardinality at most $\frac{\varepsilon}{2} N_{i_m}$ such that for $n\in [0,N_{i_m}]\setminus A_m$, 
        \[
            h_{\Omega(n)} x\in \left\{h_t x\;:\; |t- \log \log N_{i_m}|<r_m\sqrt{\log \log N_{i_m}}\right\}.
        \]
        Hence, by our assumption, the segment of the horocycle orbit $\{h_{\Omega(n)}x \,:\, n\in [0,N_{i_m}]\setminus A_m\}$ is trapped in one of the cusps $C_j$ for some \( j \leq k \). Since $r_m$ tends to infinity and $\varepsilon>0$ is arbitrary, it follows that 
        \[
            \lim_{m \to \infty} \frac{1}{N_{i_m}}\sum_{n\leq N_{i_m}} \delta_{h_{\Omega(n)}x} = \nu^j_\infty.
        \]
        Hence $\kappa=\nu^j_\infty$ for some $j\leq k$.\\
        
        \noindent \textbf{Case II.} There exists a constant $L>0$ and subsequences $(N_{i_m})\sse (N_i)$ and $ (t_m)\sse \R$ such that for all \( m \in \N \),
        \[
            t_m \in \log \log N_{i_m}+ [-L\sqrt{\log \log N_{i_m}},L\sqrt{\log \log N_{i_m}}]
        \] 
        and $h_{t_m}x\in B(e,L)$. In this case, we start with the following easy lemma:
        
        \begin{lemma}
        \label{lem:sep}
        There exists $\phi\in C^\infty_c(X)$ such that $\kappa(\phi)\neq \mu_X(\phi)$.
        \end{lemma}
        
        \begin{proof} 
            Let $A\subset X$ and $\eta_0>0$ be such that $|\kappa(A)-\mu_X(A)|>\eta_0$. If $\kappa(B(e,n))\to 1$ as $n\to\infty$, then for $n$ large enough, we have 
            \[
                \left|\kappa(A\cap B(e,n))-\mu_X(A\cap B(e,n))\right|>\eta_0/2.
            \]
            It follows that there exists a function $\phi\in C^\infty(B(e,n))$ such that $|\kappa(\phi)-\mu_X(\phi)|>\eta_0/4$. On the other hand, suppose $\kappa(B(e,n)) < 1-\eta_1$ for some $\eta_1 > 0$ and all sufficiently large $n$. Then taking a $C^\infty$ function $0\leq \phi\leq 1$ that is $0$ outside $B(e,n+1)$ and $1$ on $B(e,n)$, we have $\kappa(\phi)<\kappa(B(e,n+1))<1-\eta_1$ and $\mu_X(\phi)\geq 1-\frac{\eta_1}2$ for large enough $n$.
        \end{proof}
        
        Let \( \phi \in C_c^\infty (X)\) be that guaranteed by \cref{lem:sep}. Without loss, assume $\|\phi\|_\infty\leq 1$ and set $\phi_0=\phi-\mu_X(\phi)$. It follows that there exists $\varepsilon_0=\varepsilon_0(\phi,\kappa)>0$ such that for all $m$ large enough, 
        $$
        \left|\frac{1}{N_{i_m}}\sum_{n\leq N_{i_m}} \phi_0(h_{\Omega(n)}x) \right|\geq \varepsilon_0.
        $$
        Denote $N=N_{i_m}$ (for simplicity). For \( C > 0 \), set 
        \[
            I_N(C)=[-C\sqrt{\log \log N}+\log\log N, \log \log N +C\sqrt{\log\log N}] 
        \]
        and let $C_0=C(\varepsilon_0)>0$ be that guaranteed by Theorem \ref{thm: Hardy Ramanujan}. Then
        $$
            \left|\sum_{\ell \in I_N} \frac{\pi_N(\ell)}{N}\phi_0(h_{\ell}x)\right|\geq \varepsilon_0/2,
        $$
        where $I_N=I_N(C_0)$ and \( \pi_N(\ell) \) is as defined in \cref{subsec: Number Theory}. Let $G_N(\ell)=e^{-\frac{(\ell-\log\log N)^2}{2\log\log N}}$. Then by \cref{lem: Erdos piN Approx}, it follows that 
        $$
            \left|\frac{1}{\sqrt{2\pi\log\log N}}\sum_{\ell \in I_N}G_N(\ell)\phi_0(h_{\ell}x)\right|\geq \varepsilon_0/4.
        $$
        We now partition the interval $I_N$ into $C_1=C_1(\varepsilon_0)>0$ intervals $\{J_j\}_{j=1}^{C_1}$ of length \( c_2 \sqrt{\log \log N} \), where \( c_2 = c_2(\varepsilon_0) \) is such that for any $\ell,\ell'\in J_j=[u_j,u_{j+1})$, we have 
        \[
            |G_N(\ell)-G_N(\ell')|<\left(\frac{16 C_0}{\sqrt{2\pi}}\right)^{-1}\varepsilon_0.
        \]
        Set $G_N(j)=G_N(u_j)$. Then 
        $$
        \left|\frac{1}{\sqrt{2\pi\log\log N}}\sum_{j\leq C_1} G_N(j)\sum_{\ell \in J_j} \phi_0(h_\ell x) \right|\geq \varepsilon_0/8.
        $$
        From this, we get 
        $$
        \left|\frac{2C_0}{\sqrt{2\pi}} \cdot\frac{1}{C_1}\sum_{j\leq C_1} G_N(j) \left(\frac{1}{|J_j|}\sum_{\ell \in J_j} \phi_0(h_\ell(x))\right)\right|\geq \varepsilon_0/8.
        $$
        Hence, there exists $j \leq C_1$ such that 
        $$
        \left|\frac{1}{|J_j|}\sum_{\ell\in J_j} \phi_0(h_\ell(x))\right|\geq (16C_0)^{-1} \varepsilon_0.
        $$
        Let $\varepsilon_0'= (16C_0)^{-1} \varepsilon_0$ and let $J=J_j=[A,B]$. By construction, $A=\log \log N +z \sqrt{\log \log N} $, where $|z|< C_0$ and $B-A= c_2 \sqrt{\log \log N}$. Then
        $$
        \left|\frac{1}{B-A}\sum_{\ell=A}^B \phi_0(h_\ell x)\right|\geq \varepsilon'_0.
        $$
        \cref{thm:discSFF} then implies that for some constant $C'>0$, we have $C' \|\phi\|_{C^4} r^{-\alpha}> \varepsilon'_0$, where $r =(B-A) \cdot e^{-dist(g_{\log(B-A)}(h_{A}x),e)}$. In particular, there exists $c=c(\Gamma,\alpha,\phi, \varepsilon_0)>0$ such that $r<c^{-1}$. By the assumption on $t_m$ and $A$,  $|\frac{A-t_m}{B-A}|< C(\varepsilon_0,L)$. Combining this with the triangle inequality, we obtain
        $$
        dist(g_{\log(B-A)}(h_{A}x),e)= dist(g_{\log(B-A)}(h_{A-t_m} (h_{t_m}x),e)=$$$$dist(h_{\frac{A-t_m}{B-A}}g_{\log(B-A)}(h_{t_m}x),e)\leq 
        C'+\log(B-A).
        $$
        It follows that $r\geq e^{-C'}$, where $C'$ depends on $\varepsilon_0$ and $L$ only.
        
        Let $1>\eta>0$ be given. We now apply \cref{lem:str} with $\delta < \eta$ to be specified later, $T=R=B-A$, and $K = \frac{1}{2}\sqrt{\eta^2\delta R}$. It follows that there exists $t_0\in [0,B-A]$ and a periodic point $p\in X$ with the period $P\ll r\leq c^{-1}$ such that 
        $$
        \sup_{0\leq t\leq \frac{1}{2}\sqrt{\eta^2\delta R}} \;\; d_X(h_{t_0+t}h_{A}x,h_tp)\leq \delta.
        $$
        By the second part of \cref{lem:str} and our lower bound on $r$, we obtain $P \gg \eta^2 e^{-C'}$. Since the period of $p$ is bounded away from $0$, it follows that there exists $\delta_0=\delta_0(\eta, \varepsilon_0, L)$ such that if $\delta < \delta_0$ then (up to taking the closest lift of $p$ to $G$), 
        $$
        \sup_{0\leq t\leq \frac{1}{2}\sqrt{\eta^2\delta R}} \;\; d_{G}(h_{t_0+t}h_{A}x,h_tp)\leq \delta.
        $$
        By \cref{lem:ttadd}, we obtain
        \begin{equation}\label{eq:cc}
        h_{t_0+A}x=\gamma p\begin{pmatrix}a&0\\c&a^{-1}\end{pmatrix},
        \end{equation}
        where $|a-1|\ll (\eta^2\delta R)^{-1/2}$ and $|c|\ll  \eta^{-2}R^{-1}$ (we denote by $p$ the point $h_vp$, as it is periodic of the same period). Moreover, by \cref{lem:ntc}, it follows that $|c|(t_0+A)\to \infty$.
        
        Write \( p = g_s(e_i)\) for some \( i \leq k\). Additionally, we have 
        \[
        t_0+A=\log \log N +z'\cdot\sqrt{\log\log N},
        \]
        where $|z'|<C_0+L$. 
        Recall that \( N = N_{i_m}\), and so each of the above parameters in fact depends on the sequence \((i_m)\). Then applying \cref{lem:fini} to the sequence $(N_{i_m})$ for large enough \(m\),
        $$
            \frac1{N_{i_m}} \sum_{n \leq N_{i_m}} \delta_{h_{\Omega(n)}x}(\psi) = \left(\frac{1}{\sqrt{2\pi}}\int_{-\infty}^{\infty}e^{-\frac{r^2}{2}}\nu^{i}_{s_m+ 2\log a_m -2\log|1+c_m(r-z_m')|} dr\right)(\psi) + o_{m \to \infty}(1).
        $$
        Moreover, by the above bounds on these parameters, there exists a constant $C''(\varepsilon_0,L)$ such that  $\max(|s_m|,|c_m|,|z'_m|)\leq C''(\varepsilon_0,L)$. Hence, possibly taking a further subsequence, we may assume that each of them converge to parameters $s'',c'',z''$. Since $\varepsilon$ is arbitrary, it follows that the corresponding sequence of measures converges along the sequence $(i_m)$ to the limiting measure 
        \begin{equation}\label{eq:pi}
            \frac{1}{\sqrt{2\pi}}\int_{-\infty}^{\infty}e^{-\frac{r^2}{2}}\nu^{i}_{s''+2\log(a'')-2\log|1+c''(r-z'')|} dr.
        \end{equation}
        It then follows that the above measure is equal to $\kappa$, finishing the proof in \textbf{Case II}.

    \subsection{Proof of \cref{thm:horoboun}}  
        We will first show that under the assumptions of the theorem, {\bf Case I} from the previous subsection never holds. To do so, we use the following non-divergence result of Lindenstrauss and Mohammadi \cite{Linden-Moham}:
        
        \begin{prop}[(\cite{Linden-Moham})]
        \label{prop:lin-moh}
            There exists $C>1$ with the following property: let $0<\varepsilon,\eta<1$ and $\bar{x}\in X$. Let $I\subset [-10,10]$ be an interval with $|I|\geq \eta$. Then 
            $$
                \Big| \{r\in I\;:\; {\rm inj}(g_th_r \bar{x})<\varepsilon^2\}\Big| < C \varepsilon |I| 
            $$
            for $t\geq |\log(\eta^2{\rm inj}(\bar{x}))|+C$.
        \end{prop}

        \begin{proof}[Proof of \cref{thm:horoboun}]
             Let $C_1 > 1$ be the constant guaranteed by \cref{prop:lin-moh} and let \( T> 0\). Suppose \( x \in X \) is a point with bounded geodesic orbit and consider $\bar{x}= g_{-\log T}x$. By assumption on \(x\), there exists \( C_2 > 0\) such that ${\rm inj}(\bar{x})<C_2$. Now, take $\varepsilon=(2C_1)^{-1}$ and $\eta = C_3 T^{-1/2}$, where \( C_3 \) satisfies \( 2 \log C_3 \geq \log C_2 + C_1 \).  Notice that \( \eta < 1\) for $T$ large enough. By assumption on \( C_3\), we have
            \begin{align*}
                |\log(\eta^2{\rm inj}(\bar{x}))|+C_1 &\leq |\log(\eta^2)|+ \log C_2 +C_1
                \\
                &=  
                \log T + \log C_3^{-2}+\log C_2 + C_1
                \\
                &\leq 
                \log T.
            \end{align*}
            Finally, letting $I=[1,1+ \eta]$ and applying \cref{prop:lin-moh} with \( t = \log T\), it follows that ${\rm inj}(g_t h_r\bar{x}) \geq \frac{1}{4C_1^2}$ for at least half of the $r\in I$. In particular, ${\rm inj} (h_{rT}x)>\frac{1}{4C_1^2}$ for some $r\in I$. Summarizing, for $T>0$ large enough, there exists $t' \in [T,T+C_3T^{1/2}]$ such that ${\rm inj} (h_{t'}x)>\frac{1}{4C_1^2}$. Recall that $(N_i) \sse \N$ is a sequence for which 
            \(
            \frac{1}{N_i}\sum_{n\leq N_i} \delta_{h_{\Omega(n)}x}\to \kappa.
            \)
            Then, taking $T=\log \log N_{i}$, we see that {\bf Case I} never holds. This implies that $\nu^j_\infty\notin Acc^\Omega(x)$ for every $j\leq k$.
        
            It now remains to show that for any $j \leq k$ and any $s\in \R$, $\nu^j_s\notin Acc^\Omega(x)$. In previous subsection, we showed that for any point satisfying \textbf{Case II}, the corresponding measure takes a form as in Equation \eqref{eq:pi}. This measure is equal to $\nu^j_s$ only when $c''=0$. Hence the corresponding sequence $c'_{m}\to 0$ as $m\to \infty$. Recall that $c'_{m} = c_m \sqrt{\log \log N_{i_m}}$, where $c_m$ arises from Equation \eqref{eq:cc} with $t_0+A\leq 2\log \log N$. By \cref{lem:low}, there exists a constant \( d > 0 \), not depending on \( m \), for which $|c_m|\geq d (\log \log N_{i_m})^{-1/2}$. It follows that $|c'_{_m}| \geq d$, and so $c''\geq d > 0$. This finishes the proof of Theorem \ref{thm:horoboun}.
        \end{proof}


    \subsection{Proof of Equation \eqref{eq:subs'}}
        Let \( x \in X \) be a non-periodic point for \((h_t)\). Then \( Acc(x) = \{\mu\}\), so that by Theorem \ref{thm:main}, $\mu\in Acc^\Omega(x)$. 
        
        By Proposition 1.1 in \cite{Ratner}, it follows that a point $x\in X$ is non-periodic for $(h_t)$ if and only if the orbit $\{g_sx\}_{s>0}$ is recurrent, i.e. there exists a compact set $K_x \subset X$ and a sequence of times $(T^x_\ell)_{\ell \in \N}$ such that $g_{\log T^x_\ell}x\in K_x$. Let $d'=\min_{z\in K} e^{-dist(z,e)}$. Let $D>0$ be large enough in terms of $d'$. Since the initial point $x\in X$ is fixed in this section, we will drop the dependence on \(x \) from the return times $\{T^x_\ell\}$ and set \( K_x\). By renaming the return times, we assume they are of the form $DT_\ell$,  i.e. 
        \begin{equation}\label{eq:ret}
        g_{\log(DT_\ell)}x\in K
        \end{equation}
        
        We now fix a periodic point $p=g_{s_0}(e_i)$ for some $s_0\in \R$ and $i\leq k$. Let \( C = C_p > 10\) be the constant guaranteed by \cref{lem:perp}. In the following lemma, we approximate a piece of the horocycle orbit of $x$ by a periodic orbit. 

        \begin{prop}
        \label{prop:per} 
            There exists $D = D_{\Gamma,x}>0$ such that the following holds: there exists $\ell_P > 0$ such that for every $\ell\geq \ell_P$, there exists $t_\ell\in [T_\ell,DT_\ell]$  and $\gamma\in \Gamma $ such that 
            \begin{equation}\label{eq:per}
            \gamma x \begin{pmatrix}1&t_\ell\\0&1 \end{pmatrix} = p\begin{pmatrix}a&0\\c&a^{-1} \end{pmatrix},
            \end{equation}
            where $|a-1|<1/10$ and $c \in \left[ \frac{C^{3/2}}{2} T_\ell^{-1/2} , 4 C^{3/2} T_\ell^{-1/2}\right]$.
        \end{prop}
        
        \begin{proof}
            For \( T > 0 \), set 
            \[
                V_{T}(p)
                =
                \left\{\Gamma p\begin{pmatrix}a&0\\c&a^{-1} \end{pmatrix} \;:\; |a-1|<1/10 \text{ and }c\in \left[\frac{C^{3/2}}{2}T^{-1/2},4C^{3/2}T^{-1/2}\right]\right\}.
            \]
            The proposition is then equivalent to showing that for every non-periodic $x\in X$, there is $\ell_P>0$ such that for every $\ell\geq \ell_P$, there exists $t_\ell \in [T_\ell, D T_\ell]$ such that $h_{t_\ell}(x)\in V_{T_\ell}(p)$. For simplicity, we denote $T_\ell$ by $T$. Applying \cref{lem:perp} for the time $C^{-1}T$, there exists $s\in [C^{3/2}T^{-1/2},2C^{3/2}T^{-1/2}]$ such that $g_{\log(T)}(v_s(p))\in B(e,1)$. 
        
            Notice that $\left\{g_bv_{s+d}(p))\;:\; |e^{b/2}-1|<1/10, |d|\leq \frac{1}{T}\right\}\subset V_{T}(p)$. This follows from the fact that 
            $$
                \begin{pmatrix}1&0\\s+d&1 \end{pmatrix} \cdot \begin{pmatrix}e^{b/2}&0\\0&e^{-b/2}\end{pmatrix}
                = 
                \begin{pmatrix}e^{b/2}&0\\(s+d)e^{b/2}&e^{-b/2}\end{pmatrix},
            $$
            and $(s+d)e^{b/2}\in[\frac{C^{1/2}}{2}T^{-1/2},4C^{1/2}T^{-1/2}]$. Then it is enough to  show that there exists 
            $t_\ell\in [T_\ell,DT_\ell]$ such that $h_{t_\ell}(x)\in \{g_bv_{s+d}(p))\;:\; |e^{b/2}-1|<1/10, |d|\leq \frac{1}{T_\ell}\}$. 
            
            Composing with $g_{\log T}$ and using the renormalization equations \eqref{eq:norm}, this in turn is equivalent to showing that there exists $\bar{t}\in [1,D]$ such that 
            \[
                h_{\bar{t}}(g_{\log T}(x))\in G = G(s,p,T):=\{g_bv_d(g_{\log T}v_sp)\;:\; |b/2-1|<1/10, 0\leq d\leq 1 \}.
            \]
            Let $\tilde{G}=\bigcup_{t\in (0,1)}h_t(G)$. Notice that $\mu_X(\tilde{G})\geq d_0$ for some $d_0$ depending only on $\Gamma$. In particular, we need to show that there is $D>0$ such that
            \begin{equation}\label{eq:chit}
            \int^{D}_{1} \mathbbm{1}_{\tilde{G}}(h_t(g_{\log T}(x)))dt>0.
            \end{equation}
            \cref{thm:SFF} and \cref{lem:approx} imply that
            $$
            \left|\frac{1}{D}\int^{D}_{1} \mathbbm{1}_{\tilde{G}}(h_t(g_{\log T}(x)))dt-\mu(\tilde{G})\right|
            \leq 
            C' C^{5} \bar{r}^{-\alpha},
            $$
            where $\bar{r}=D\cdot e^{-dist(g_{\log D}(g_{\log T}(x)),e)}$. Note that by Equation \eqref{eq:ret}, $g_{\log(DT)}x\in K$, and so $\bar{r}\geq Dd'\geq D^{1/2}$ for \( D \) large enough. Plugging this into the equation above, we obtain $\frac{1}{D}\int^{D}_{1} \mathbbm{1}_{\tilde{G}}(h_t(g_{\log T}(x)))dt>0$, as desired.
        \end{proof}
        
        Using the set-up from this subsection, we now give a proof of Equation \eqref{eq:subs'}. 
        
        \begin{proof}
            Given $x \in X$ non-periodic and $p = g_{s_0}(e_i)$ for some \( s_0 \in \R , i \leq k \), let $(t_\ell)$ be the sequence given by \cref{prop:per}. We now argue similarly as in \cref{subsec:eqn1 proof}. Let $M_\ell=2^{2^{t_\ell}}$, so that $t_\ell=\log \log M_\ell$.

            Given $\varepsilon > 0$, let $C_\varepsilon$ that guaranteed by \cref{thm: Hardy Ramanujan}. Let $I_\ell$ denote the interval 
            \[
                I_{M_\ell}(C_\varepsilon) 
                =
                \log \log M_\ell + [  - C_\varepsilon \sqrt{\log \log M_\ell}\,,\,  C_\varepsilon \sqrt{\log \log M_\ell}] 
                =
                t_\ell + [-C_\varepsilon t_\ell^{1/2}, \, C_\varepsilon t_\ell^{1/2}].
            \]
            and for $j\in \Z$, set $K_{j}= [j\varepsilon^4 t_\ell^{1/2},(j+1)\varepsilon^4t_\ell^{1/2})$. We consider intervals $\{K_j\}_{j=-U}^U$, where $U=C_\varepsilon \varepsilon^{-4}$. Notice that \( t_\ell + K_j \) partitions \(I_\ell\). Then 
            \begin{align} \label{eq:approx}
                \notag \frac{1}{M_\ell}\sum_{n\leq M_\ell}\psi(h_{\Omega(n)}x)
                &= 
                \frac{1}{2\pi \cdot t_\ell^{1/2}}  \sum_{t \in I_\ell} G_{t_\ell}(k) \psi(h_t x) + o(1)
                \\
                &= 
                \frac{2C_\varepsilon}{\sqrt{2\pi}\cdot 2U}\sum_{j=-U}^U \frac{G_N(j)}{|K_j|}\sum_{t \in K_j}\psi(h_{t}(h_{t_\ell}x))+ o(1),
            \end{align}
            where $G_{t_\ell}(j) = G_{t_\ell}(k_j) = e^{k_j^2/(2 t_\ell)}$ and $k_j$ is the right endpoint of $K_j$. Since \( t_\ell \) was chosen by \cref{prop:per}, there exists \( \gamma \in \Gamma\) such that 
            \[
                \gamma x \begin{pmatrix}1&t_\ell\\0&1 \end{pmatrix} = p_a \begin{pmatrix}1&0\\ac&1 \end{pmatrix},
            \]
            where $p_a=g_{2\log a}p$, $|a-1| < 1/10$, and \(c \in \left[ \frac{C^{1/2}}{2} T_\ell^{-1/2} , 4 C^{1/2} T_\ell^{-1/2}\right] \). Consider those $j\in [-U,U]$ for which $|1+act|<\varepsilon$ for some $t\in K_j$. Since $t_\ell^{-1/2}\ll |ac|\ll t_\ell^{-1/2}$, the union of all such $\{K_j\}$ has length $\ll \varepsilon \sqrt{t_\ell}$, and so they can be discarded. 

            Now, applying \cref{lem:tadd} with \( y = h_{t_\ell} x\), \(J = K_j = [u_j, u_{j+1})]\), and \( R = t_\ell^{1/2}\), we obtain 
            \[
                \sup_{t \in K_j} d_X \left( h_{\frac{t}{1+act}(1+acu_j)^2}(g_{-2\log|1+acu_j|}p_a), h_{t+t_\ell}x\right) \ll \varepsilon.
            \]
            Then \cref{lem:ns} implies that
            \begin{align*}
                \frac{1}{|K_j|} \sum_{t \in K_j}\psi(h_{t}(h_{t_i}x))
                &=
                \frac{1}{|K_j|} \sum_{t \in K_j}\psi\left(h_{\frac{t}{1+act}(1+acu_j)^2}g_{-2\log|1+acu_j|}p_a\right)+ o_{\varepsilon \to 0}(1)\\
                &= \nu^i_{s_0+2\log a-2\log|1+acu_j|}(\psi) +o_{j \to \infty}(1)+o_{\varepsilon \to 0}(1).
            \end{align*}
            Combined with Equation \eqref{eq:approx}, this implies 
            $$
                \frac{1}{M_i}\sum_{n\leq M_i}\psi(h_{\Omega(n)}x)=\frac{2C_\varepsilon}{\sqrt{2\pi}\cdot 2U}\sum_{j=-U}^U G_{t_\ell}(j)\nu^i_{s_0+2\log a-2\log|1+cu_j|}(\psi)+ o_{\varepsilon \to 0}(1).
            $$
            Now, set \( c' = ac \, t_{\ell}^{1/2}\). Then, using the assumptions on \( c\) and \(t_\ell\), we see that $c'\in [\frac{C^{3/2}}{2D}, 4C^{3/2}]$. Moreover, by the definition of Riemann integral, 
            \begin{align*}
                \frac{2 C_\varepsilon}{\sqrt{2\pi}\cdot 2U}\sum_{j=-U}^U G_{t_\ell}(j)&\nu^i_{s_0+2\log a+2\log|1+acu_j|}(\psi)\\
                &= 
                \frac{2C_\varepsilon}{\sqrt{2\pi}\cdot 2U}\sum_{j=-U}^U G_{t_\ell}(j)\nu^i_{s_0+2\log a + 2\log|1+c'j\varepsilon^4|}(\psi)
                \\
                &= 
                \frac{1}{\sqrt{2\pi}}\int_{-\infty}^\infty e^{-r^2/2}\nu^i_{s_0+2\log a + \log|1+c'r|} dr+ o_{\varepsilon \to 0}(1).
            \end{align*}
            Notice that the parameters \(a \) and \( c \) in fact depend on the subsequence \( t_\ell\). Summarizing, for any $\varepsilon>0$, there exists $M_\varepsilon$ such that for $M_\ell \geq M_\varepsilon$, there is some $c_\ell'\in [\frac{C^{3/2}}{2D}, 4C^{3/2}]$ and $s'_{\ell}\in [s_0 - \frac12,s_0+ \frac12]$ such that 
            $$
                \left|\frac{1}{M_\ell}\sum_{n\leq M_\ell}\psi(h_{\Omega(n)}x)-\Big(\frac{1}{\sqrt{2\pi}}\int_{-\infty}^\infty e^{-r^2/2}\nu^i_{s'_\ell-2\log|1+c'_\ell r|} dr\Big)(\psi)\right|\ll \varepsilon.
            $$
            Up to a subsequence, both $(c'_\ell)$ and $(s'_{\ell})$ converge to some $c'\in [\frac{C^{3/2}}{2D}, 4C^{3/2}]$ and $s' \in [s_0 - \frac12,s_0+ \frac12]$. Then, possibly taking a further subsequence, we obtain 
            $$
                \lim_{\ell \to \infty} \frac{1}{M_{\ell}}\sum_{n\leq M_{\ell}}\delta_{h_{\Omega(n)}x}=\frac{1}{\sqrt{2\pi}}\int_{-\infty}^\infty e^{-r^2/2}\nu^i_{s'-2\log|1+c' r|} dr.
            $$
            Taking \( D' = 8D\), this finishes the proof.
        \end{proof}

\section{Proof of Theorem \ref{thm:horoubo}} 

    Note that by Theorem \ref{thm:main}, $\mu_X \in Acc^\Omega(x)$ for any $x\in X$. The theorem follows from the following two lemmas:
    
    \begin{prop}\label{lem:nr1} 
        For any $i\leq k$, there exists a dense $G_\delta$-set of points $G_i\subset X$ such that for every $x\in G_i$,  $\nu^i_\infty \in Acc^\Omega(x)$.
    \end{prop}

    \begin{prop} \label{lem:nr2} 
    For any $i\leq k$, $s, c, z\in \R$ with $1-cz\neq 0$, there exists a dense $G_\delta$-set of points $G\subset X$ such that
    $$ 
    \frac{1}{\sqrt{2\pi}}\int_{-\infty}^{\infty}e^{-\frac{r^2}{2}}\nu^{i}_{s+2\log|1+c(r-z)|} dr\in Acc^\Omega(x).
    $$

    \end{prop}

    Indeed notice that it is enough to take a countable dense set in $\{(z,c,s)\;:\; z,c,s\in \R \}$ and apply the dominated convergence theorem and Baire category theorem together with Propositions \ref{lem:nr1} and \ref{lem:nr2}.

    \subsection{Proof of Proposition \ref{lem:nr1}}
    
        Recall that $C_j$ denotes the $j$-th cusp for \( j \leq k \), where \(k\) denotes the number of inequivalent cusps. Let $w(t)$ be a fixed positive function such that $w(t)\cdot t\to \infty$ as $t\to +\infty$. We prove the following proposition:
        \begin{prop}\label{prop:cusp}
            For each \( j \leq k \), there exists a dense $G_\delta$-set $G_j\subset X$ such that for every $x\in G_j$, there exists a sequence of times $(T_i)_{i \in \N}$ such that 
            $$
                h_{s}x\in C_j \setminus K\text{ for } s\in [T_i-w(T_i)^{-1}, T_i+w(T_i)^{-1}]
            $$
            for every compact set $K\subset X$ and sufficiently large $i\geq i_K$.
        \end{prop}

        \begin{proof}[Proof of \cref{lem:nr1}]
             Let $w(t)=T^{-1/2-1/100}$ and let \( (T_i) \) be that guaranteed by \cref{prop:cusp}. Taking $M_i=e^{e^{T_i}}$ implies that $h_{\Omega(s)}x\in  C_j \setminus K$ for all 
             \[
                s\in [T_i-T_i^{1/2+1/100}\,,\,T_i+T_i^{1/2+1/100}],
            \]
            except on a set of density going to $0$ (see Theorem \ref{thm: Hardy Ramanujan}). This implies Proposition \ref{lem:nr1}.
        \end{proof}

        The main tool in proving \cref{prop:cusp} is the following lemma:
        
        \begin{lemma}\label{lem:cusp} 
            Let $\varepsilon>0$ and $B\subset X$ an open set. Then there exists $T_0=T_0(\varepsilon,B)>0$ such that for every $T>T_0$ and every periodic point $p$ of period $\geq \varepsilon$, we have 
            $$
                h_{-T}\Big(\{v_s(p)\;:\; |s|<w(T)\}\Big)\cap B \neq \emptyset.
            $$
        \end{lemma}

        Before proving \cref{lem:cusp}, we show how it implies Proposition \ref{prop:cusp}:
        
        \begin{proof}[Proof of Proposition \ref{prop:cusp}]
            Consider the open set
            $$
            H(T,p)=\bigcup_{|a|,|u|< 1}\{v_s(g_{a}h_up)\;:\; |s|<w(T)\}.
            $$
            Note that 
            \[
                \{v_s(p)\;:\; |s|<w(T)\}\subset H(T,p).
            \]
            Let $(p_\ell)$ be a sequence of periodic points of period going to $0$. Define 
            \[
                C :=\bigcap_{\ell,m\in \N} \left( \bigcup_{T>m, T\in \N}h_{-T}(H(T,p_\ell))\right).
            \]
            Note that $C$ is a $G_\delta$-set. Moreover, by Lemma \ref{lem:cusp}, for every ball $B$ and every $j,m \in \N$, there exists $T>m$ such that 
            \[
                h_{-T}(H(T,p_\ell))\cap B\neq \emptyset.
            \]
            This implies that $C$ is dense. Note that for $x\in C$, it follows that there exists a sequence of times $(T_i)$ and periodic points $\tilde{p}_i = g_{a_i} h_{u_i} p_{\ell_i}$ of period going to $0$ such that 
            \[
                h_{T_i}x\in \{v_s(\tilde{p}_i)\;:\; |s|< w(T_i)\}.
            \]
            This then implies that for every $|r|\leq w(T_i)^{-1}$, $h_{r+T_i}(x)=h_r(v_s(\tilde{p}_i))=h_r(v_s(h_{-r})(h_r(\tilde{p}_i))$. It remains to notice that
            $$
            \begin{pmatrix}1&r\\0&1\end{pmatrix}\begin{pmatrix}1&0\\s&1\end{pmatrix}\begin{pmatrix}1&-r\\0&1\end{pmatrix}=\begin{pmatrix}1+rs&-rs^2\\
            s&1-rs\end{pmatrix}
            $$
            and $|rs|< w(T_i)^{-1}w(T_i) \leq 1$. This implies that for every $|r|< w(T_i)^{-1}$, $h_{r+T_i}(x)$ is a bounded distance from the periodic orbit of $\tilde{p}_i$, which leaves every compact set as $i\to \infty$.
        
        \end{proof}

        \begin{proof}[Proof of Lemma \ref{lem:cusp}]
            We in fact show the statement for $s\in [\frac{k_T}{T}, \frac{k_T}{T}+ \frac{10k_T}{T^2}]$, where $k_T>0$ is a parameter to be specified later satisfying $k_T\to \infty$ and $\frac{2k_T}{T}\leq w(T)$. Let  $\bar{c}(s)=\frac{s}{1-sT}$. By a direct computation, 
            \begin{align}
            \label{eq:cusp1}
                \notag h_{-T}(v_sp) &=\gamma p\begin{pmatrix}1&0\\s&1\end{pmatrix}\begin{pmatrix}1&-T\\0&1\end{pmatrix}=\gamma p\begin{pmatrix}1&-T\\s&1-sT\end{pmatrix}\\
                &= \gamma p \begin{pmatrix}1&\frac{-T}{1-sT}\\0&1\end{pmatrix}\begin{pmatrix}(1-sT)^{-1}&0\\0&1-sT\end{pmatrix} \begin{pmatrix}1&0\\ \bar{c}(s)&1\end{pmatrix}.
            \end{align}

            Taking $s=\frac{k_T}{T}+\frac{R(s)}{T^2}$, where $R(s) \leq 10k_T$, we get that \eqref{eq:cusp1} is equal to
            \begin{align*}
                \gamma p\begin{pmatrix}(1-k_T)^{-1}&0\\0&(1-k_T)\end{pmatrix} &\begin{pmatrix}1&\frac{-T}{1-k_T- \frac{R(s)}{T}}(1-k_T)^2\\0&1\end{pmatrix}\\
                &\qquad \begin{pmatrix}\left(1+\frac{R(s)}{T(1-k_T)}\right)^{-1}&0\\0&1+\frac{R(s)}{T(1-k_T)}\end{pmatrix} \begin{pmatrix}1&0\\ \bar{c}(s)&1\end{pmatrix}.
            \end{align*}
    
            It remains to notice that with our choice of parameters, $|\bar{c}(s)| \ll T^{-1}$ and $\frac{R(s)}{T(1-k_T)} \ll T^{-1}$. Hence $h_{-T}(v_sp)$ is close to $h_{Z(T,R(s))}g_{-\log |1-k_T|}(p)$, where 
            \[
                Z(T,R(s))=\frac{-T}{1-k_T- \frac{R(s)}{T}}(1-k_T)^2.
            \] 
            Note that $Z(T,R(\cdot))$ is a continuous function of $s$, with $Z(T,0)=-T(1-k_T)$ and 
            \begin{align*}
                |Z(T,10k_T)-Z(T,0)| &= \left|T(1-k_t) - \frac{(1-k_T)^2}{1 - k_T - \frac{10k_T}{T}} \right|
                \\
                &= \frac{10k_T(1-k_T)}{1-k_T - \frac{10k_T}{T}}
                \\
                &> 5k_T
            \end{align*}
            for \( T \) large enough. Moreover, $g_{-\log |1-k_T|}(p)$ is a periodic point of period $|1-k_T|$. By a result of Sarnak \cite{Sarnak81}, the horocycle orbit of such periodic points become equidistributed as \( T \to \infty\). Hence the orbit enters the ball $B$ if $k_T$ is large enough. This finishes the proof.
        \end{proof}

        \subsection{Proof of Proposition \ref{lem:nr2}}  

        Analogously to the previous subsection, the main tool is the following lemma:

        \begin{lemma}\label{lem:gauss} 
           Let $B \subset X$ be open, let $c \in \R\setminus\{0\}$, and let $p$ be a periodic point. Then there exists $T_0 = T_0(p, B)>0$ such that for every $T>T_0$, 
            \begin{equation}\label{eq:inq}
                h_{-T}\left(\left\{v_{u+s'}(p)\;:\; |s'|<T^{-2/3}\}\right\}\right)\cap B \neq \emptyset,
            \end{equation}
            where $u=\frac{c}{\sqrt{T}}$.
        \end{lemma}
        
        The proof of this lemma is similar to the proof of \cref{lem:cusp}. We provide it here for completeness.
        
        \begin{proof}   
            Let $\bar{c}(r)=\frac{r}{1-rT}$. By a direct computation, 
            \begin{align*}
            \label{eq: cusp2}
                \notag h_{-T}(v_{r}p) &=\gamma p\begin{pmatrix}1&0\\r &1\end{pmatrix}\begin{pmatrix}1&-T\\0&1\end{pmatrix}\\
                &=\gamma p  \begin{pmatrix}(1-rT)^{-1}&0\\0&1-rT\end{pmatrix}\begin{pmatrix}1&-T(1-rT)\\0&1\end{pmatrix} \begin{pmatrix}1&0\\ \bar{c}(r)&1\end{pmatrix}\\
                &=\gamma p  \begin{pmatrix}(1-uT)^{-1}&0\\0&1-uT\end{pmatrix} \begin{pmatrix}1&-T\frac{(1-uT)^2}{1-rT}\\0&1\end{pmatrix}
                \\ 
                & \hspace{160pt} \begin{pmatrix}\frac{1-uT}{1-rT}&0\\0&\Big(\frac{1-uT}{1-rT}\Big)^{-1}\end{pmatrix} \begin{pmatrix}1&0\\ \bar{c}(r)&1\end{pmatrix}.
            \end{align*}     
            Taking \( r = u + s'\), notice that with our choice of parameters, $|\bar{c}(u+s')|\ll T^{-1}$ and $\Big| \frac{1-uT}{1-(u+s')T}-1\Big|\ll T^{-1/6}$. Let $W(s'):=-T\frac{(1-uT)^2}{1-(u+s')T}$. Then  $h_{-T}(v_{u+s'}p)$ is very close to  $h_{W(s')}(g_{-2\log(1-uT)}(p))$. Note that $W(\cdot)$ is continuous and 
            \[
                |W(0)-W(s')|=T(1-uT)^2\frac{s'T}{(1-uT)(1-uT-s'T)}\geq \frac{1}{2} s'T^2.
            \]
            Taking $s'=T^{-2/3}$, we get that the range of $W(\cdot)$ contains an interval of length $T^{4/3}$. Note that $g_{-2\log |1-uT|}(p)$ is a periodic point of period $|1-uT|^2 per(p)$. Then the period goes to infinity with $T$ and is upper bounded by $Tper(p)$. In particular, by the above computation and a result of Sarnak \cite{Sarnak81}, there exists $s'\leq T^{-2/3}$ such that $h_{W(s')}(g_{-2\log(1-uT)}(p))\in B$. This finishes the proof.
        \end{proof}

        We now prove Proposition \ref{lem:nr2} using \cref{lem:gauss}:
        
        \begin{proof}[Proof of Proposition \ref{lem:nr2}]
            Fix $s,c,z \in \R $ as in the statement. Let $s'=s-2\log(1-cz)$ and $c'=\frac{c}{|1-cz|}$. We apply Lemma \ref{lem:gauss} with $p=g_{s'}(e_i)$ and $u=\frac{c'}{\sqrt{T}}$. Let 
            \[
                H'(T,p,c')=\left\{h_\xi g_{\xi'} v_{\frac{c'}{\sqrt{T}}+c''}(p)\;:\; |c''|<\frac{1}{T^{2/3}}, |\xi|,|\xi'|<e^{T^{-100}}\right\}.
            \]
            Note that $H'(T,p,c')$ is open and it contains the set $\left\{v_{\frac{c'}{\sqrt{T}}+c''}(p)\;:\; |c''|<\frac{1}{T^{2/3}}\right\}$. Define 
            \[
                G:=\bigcap_{m\in \N} \left(\bigcup_{T>m,T\in \Z} h_{-T}(H'(T,p,c'))\right).
            \]
            Then $G$ is a $G_\delta$-set by definition. Moreover, by \cref{lem:gauss}, for every ball $B$ and every $m$, there exists $T$ such that $h_{-T}(H'(T,p,c'))\cap B\neq \emptyset$, so \( G \) is dense. For $x\in G$, there exists a sequence \( (T_j)\) such that $h_{T_j}x\in H'(T_j,p,c')$. Hence
            \[
                h_{T_j}x = \gamma p \begin{pmatrix}1&0\\\frac{c'}{\sqrt{T}}+c''& 1 \end{pmatrix}g_{\xi'} h_\xi,
            \]
            where $|c''|< \frac{1}{T^{2/3}}$ and $|\eta'|, |\eta|<e^{T^{-100}}$. We then conclude by Lemma \ref{lem:fini} that along the sequence $T_j'=2^{2^{T_j}}$, 
            $$ 
            \lim_{j \to \infty} \frac{1}{T_j'}\sum_{n\leq T_j'}\delta_{h_{\Omega(n)}x} = \frac{1}{\sqrt{2\pi}}\int_{-\infty}^{\infty}e^{-\frac{r^2}{2}}\nu^{i}_{s'-2\log|1+c'r|} dr
            $$
            It remains to notice that $s'-2\log|1+c'r|=s-2\log|1+c(r-z)|$. This finishes the proof.
        \end{proof}

    \section{Proof of Pointwise Divergence}
        In this section, we prove \cref{cor:ptwise}, demonstrating pointwise almost everywhere divergence of ergodic averages along \( \Omega(n) \). Recall that a point \( x \in X \) is non-periodic for $(h_t)$ if and only if there exists a compact set $K_x \subset X$ and a sequence of times $(T^x_\ell)_{\ell \in \N}$ such that $g_{\log T^x_\ell}x\in K_x$. Note that, excluding a set of zero measure, this compact set can be taken independent of \( x \). To see this, let \( K \subset X \) be any compact set satisfying \( \mu_X(K) > 0 \). Then, by ergodicity of the geodesic flow, 
        \[
            \lim_{T \to \infty} \frac1T \int_0^T \mathbbm{1}_K (g_s x) \, ds = \mu_X(K) > 0
        \]
        for \( \mu_X\)-almost every \( x \in X \). In particular, we can take \( D > 0 \) in \cref{thm:horo} to be independent of the non-periodic point \( x \) off a set of measure zero, say \( \mathcal{N} \subset X\).

        \begin{proof}[Proof of \cref{cor:ptwise}]
            Fix \( i \leq k\). Let \( x \in X \setminus \mathcal{N} \) be a non-periodic point for \( (h_t) \) and let \( D > 0 \) be the constant guaranteed by \cref{thm:horo}. Note that, by the above, \( D \) can be chosen independent of \( x \). Taking \( s_0 = 3/2 \), \cref{thm:horo} implies that there exists \( s_x \in [1,2] \) and \( C_1 > 0\) such that 
            \[
                \tilde{\nu}_x := \frac{1}{\sqrt{2\pi}}\int_{-\infty}^{\infty}e^{-\frac{r^2}{2}}\nu^{i}_{s_x-2\log|1+c_x r|} dr \in Acc^{\Omega}(x)
            \]
            for some \( c_x \in [C_1, D C_1].\) Note that \( C_1 \) depends only on our choice of \( s_0 \). Moreover, \( \mu_X \in Acc^{\Omega}(x)\). Then there exist subsequences \( (N_i), (M_i) \sse \N\) such that for all \( \psi \in C_c(X)\),
            \[
                \lim_{i \to \infty} \frac{1}{N_i} \sum_{ n \leq N_i} \psi(h_{\Omega(n)}x) = \mu_X(\psi) 
            \]
            and 
            \[
                \lim_{i \to \infty} \frac{1}{M_i} \sum_{m \leq M_i} \psi(h_{\Omega(n)}x) = \tilde{\nu}_x(\psi) .
            \]
            Hence, it suffices to find a function \( \psi_0\), independent of \( x \), such that \( \mu_X(\psi_0) \neq \tilde{\nu_x}(\psi_0) \).

            Notice that, since \( C_1 \) and \(D\) are each independent of \(x \), the parameters \( s_x\) and \( c_x \) are uniformly bounded in \( x\). Let \( \varepsilon > 0\) and let \( C_\varepsilon > 0\) be the constant guaranteed by \cref{thm: Hardy Ramanujan}. Then 
            \[
                \tilde{\nu}_x(\psi) = \left( \frac{1}{\sqrt{2\pi}}\int_{-C_\varepsilon}^{C_\varepsilon}e^{-\frac{r^2}{2}}\nu^{i}_{s_x-2\log|1+c_x r|} dr \right) (\psi) + {\rm O}\left(e^{-(C_\varepsilon)^2}\right).
            \]
            Let \( P_{x,r} \) denote the period of the point \( g_{s_x-2\log|1+c_x r|}(e_i)\). Then there exists \( C_2 > 0 \) such that
            \[
                \sup_{0 \leq z \leq P_{x,r}} d_X(h_z g_{s_x-2\log|1+c_x r|}(e_i), e_i) \leq 2 |\log |1+c_x r|| + C_2.
            \]
            Now, set \( I = -\frac{1}{c_x} + \left[ - \frac{e^{-C_\varepsilon}}{c_x}, \frac{e^{-C_\varepsilon}}{c_x}\right] \), and notice that for \( r \notin I\), we have \( \Big|\log |1+c_x r|\Big| <  C_\varepsilon \). Consider the set 
            \[
                B = \{ y \in X \,:\, d_{X}(y, e) \leq 2 C_\varepsilon + C_2 \}.
            \]
            By construction, 
            \begin{align*}
                \left( \frac{1}{\sqrt{2\pi}}\int_{-C_\varepsilon}^{C_\varepsilon}e^{-\frac{r^2}{2}}\nu^{i}_{s_x-2\log|1+c_x r|} dr \right) (\mathbbm{1}_B) 
                &= 
                \left( \frac{1}{\sqrt{2\pi}}\int_{r \in I}e^{-\frac{r^2}{2}}\nu^{i}_{s_x-2\log|1+c_x r|} dr \right) (\mathbbm{1}_B)
                \\
                &\asymp e^{-C_\varepsilon}.
            \end{align*}
            Hence
            \(
                \tilde{\nu}_x(\mathbbm{1}_B) \asymp e^{-C_\varepsilon}.
            \)
            However, 
            \(
                \mu_X(\mathbbm{1}_B) \asymp e^{-2C_\varepsilon}, 
            \)
            so that there exist \( C_3\), \( C_4 > 0\) such that 
            \[
                \tilde{\nu}_x(\mathbbm{1}_B) \geq C_3 e^{-C_\varepsilon} > C_4 e^{-2C\varepsilon} \geq \mu_X(\mathbbm{1}_B)
            \]
            for all non-periodic \( x \in X \setminus \mathcal{N}\). Since this set of points has full measure, taking \( f \) to be a continuous approximation of a cutoff of \( B\) finishes the proof. 
        \end{proof}

\section*{Funding}
  The first author acknowledges the support of the National Science Foundation under grant DMS-2554280 and the second author under grant DMS-2402158.

\bibliography{MainReferenceLibrary}{}
\bibliographystyle{halpha}

\end{document}